\documentclass[11pt,twoside]{amsart}
\usepackage{amsmath,amssymb,amsthm}

\newtheorem{thm}{Theorem}[section]

 \newtheorem{cor}{Corollary}[section]
 
 \newtheorem{prop}{Proposition}[section]
 \newtheorem{defn}{Definition}[section]
\newtheorem{rem}{Remark}[section]

\def\Id{{\rm Id}\,}
\def\d{\partial}
\def\ddj{\dot \Delta_j}

\def\ddk{\dot \Delta_k}

\def\tilde{\widetilde}

\def\wt{\widetilde}

\newcommand\R{\mathbb{R}}

\newcommand\Z{\mathbb{Z}}

\newcommand{\N}{\mathbb{N}}
\newcommand{\ep}{\varepsilon}

\newcommand{\Supp}{\hbox{Supp}\,}

\renewcommand{\div}{\mbox{\rm div}\;\!}

\def\cA{{\mathcal A}}

\def\cC{{\mathcal C}}
\def\cD{{\mathcal D}}

\def\cF{{\mathcal F}}

\def\cP{{\mathcal P}}
\def\cQ{{\mathcal Q}}
\def\cS{{\mathcal S}}
\def\cX{{\mathcal X}}

\newcommand{\Int}{\displaystyle \int}
\newcommand{\Frac}{\displaystyle \frac}

\newcommand{\Sup}{\displaystyle \sup}

\begin{document}
\title[Decay estimates for the compressible Navier-Stokes equations]{Optimal time-decay estimates for the compressible Navier-Stokes equations in the critical $L^{p}$ framework}
\author{Rapha\"el Danchin}
\address{Universit\'{e} Paris-Est,  LAMA (UMR 8050), UPEMLV, UPEC, CNRS, 61 avenue du G\'en\'eral de Gaulle, 94010 Cr\'eteil Cedex 10}
\email{danchin@univ-paris12.fr}
\thanks{The first author is supported by ANR-15-CE40-0011 and by the  Institut Universitaire de France.}

\author{Jiang Xu}
\address{Department of Mathematics,  Nanjing
University of Aeronautics and Astronautics,
Nanjing 211106, P.R.China,}
\email{jiangxu\underline{ }79math@yahoo.com}
\thanks{The second author  is partially supported by the National
Natural Science Foundation of China (11471158), the Program for New Century Excellent
Talents in University (NCET-13-0857) and the Fundamental Research Funds for the Central
Universities (NE2015005). He would like to thank Professor A. Matsumura for  introducing  him
to  the decay problem for partially parabolic equations when he visited Osaka University.  He is also grateful to
Professor R. Danchin for his kind hospitality when visiting the LAMA in UPEC}

\subjclass{76N15, 35Q30, 35L65, 35K65}
\keywords{Time decay rates; Navier-Stokes equations; critical spaces;  $L^p$ framework.}

\begin{abstract}
The global existence issue for the isentropic compressible
Navier-Stokes equations in the critical regularity framework has been  addressed in \cite{D0} more than
fifteen years ago.
However,   whether (optimal) time-decay rates
could be shown in general critical spaces and any dimension $d\geq2$  has remained an open question.
Here we give a positive answer to that issue not only in the $L^2$ critical framework of \cite{D0} but also in the more general
 $L^p$ critical framework  of \cite{CD,CMZ2,H2}.
 More precisely, we show that under a mild additional decay assumption
 that is satisfied if the low frequencies of the initial data are in e.g. $L^{p/2}(\R^d)$, the $L^p$ norm
(the slightly stronger $\dot B^0_{p,1}$ norm in fact) of the critical global solutions
 decays like $t^{-d(\frac 1p-\frac14)}$ for $t\to+\infty,$
exactly as firstly observed by A. Matsumura and T. Nishida in \cite{MN1} in the  case $p=2$ and $d=3,$
 for  solutions with high  Sobolev regularity.

Our method relies on refined time weighted inequalities in the Fourier space, and
is likely to be effective   for other  hyperbolic/parabolic systems that are encountered
in fluid mechanics or mathematical physics.
\end{abstract}

\maketitle

\section{Introduction}\setcounter{equation}{0}
In Eulerian coordinates, the motion of a general barotropic compressible fluid in the whole space $\R^d$ is governed by the
following Navier-Stokes system:
\begin{equation}\label{R-E1}
\left\{\begin{array}{l}
\partial_t\varrho+\div(\varrho u)=0,\\[1ex]
\partial_t(\varrho u)+\div(\varrho u\otimes u)-\div\bigl(2\mu D(u)+\lambda\,\div u\, \Id\bigr)+\nabla\Pi=0.
\end{array}\right.
\end{equation}
Here $u=u(t,x)\in \mathbb{R}^{d}$ (with  $(t,x)\in \mathbb{R}_{+}\times \mathbb{R}^{d}$)  stands for the velocity field
and $\varrho=\varrho(t,x)\in \mathbb{R}_{+}$ is the density. The barotropic assumption means that
 $\Pi\triangleq P(\varrho)$ for some given function $P$ (that will be taken suitably smooth in all that follows).
 The notation  $D(u)\triangleq\frac12(D_x u+{}^T\!D_x u)$ stands for  the {\it deformation tensor}, and $\div$
 is the divergence operator with respect to the space variable.
 The  density-dependent functions $\lambda$ and $\mu$  (the  \emph{bulk} and \emph{shear viscosities})
 are supposed to be smooth enough and to satisfy
 \begin{equation}\label{eq:pos-visc}
 \mu > 0\quad\hbox{and}\quad \nu\triangleq\lambda+2\mu>0.
 \end{equation}
System \eqref{R-E1}
is supplemented with initial data
\begin{equation}\label{R-E2}
(\varrho,u)|_{t=0}=(\varrho_0,u_0),
\end{equation}
and we focus on solutions going to some constant state $(\varrho_\infty,0)$ with $\varrho_\infty>0,$ at infinity.
\medbreak
As in many works dedicated to nonlinear evolutionary PDEs, \emph{scaling invariance} 
will play a fundamental role in our paper. The reason why is that  whenever such an invariance exists,
  suitable critical quantities (that is, having the
same scaling invariance as the system under consideration) control the possible finite time blow-up, and
the global existence of strong solutions.
In our situation, we observe that  \eqref{R-E1} is invariant by the transformation
 \begin{equation}\label{scaling}
\varrho(t,x)\leadsto \varrho(\ell ^2t,\ell x),\quad
u(t,x)\leadsto \ell u(\ell^2t,\ell x), \qquad\ell>0,
\end{equation}
up to a change of the pressure term $\Pi$ into $\ell^2\Pi.$
Therefore we expect critical norms or spaces for investigating \eqref{R-E1}, to have the scaling invariance \eqref{scaling}.
\medbreak
 As observed  by the first author in \cite{D0}, one may solve
  \eqref{R-E1} in critical  homogeneous Besov spaces of type $\dot B^s_{2,1}$
  (see Def. \ref{d:besov}). In that context, in accordance with \eqref{scaling}, with the scaling properties of Besov spaces (see \eqref{eq:scaling} below) and with the conditions at $\infty,$   criticality  means that
   $a_0\triangleq\varrho_0-\varrho_\infty$ and $u_0$ have to be taken in  $\dot B^{d/2}_{2,1}$ and  $\dot B^{d/2-1}_{2,1},$ respectively.
Besides smallness however, in order to achieve \emph{global existence},
an additional  condition has to be prescribed
on the low frequencies  of the density. This  comes from   the fact that the scaling invariance in \eqref{scaling} modifies the (low order) pressure term.
Schematically,  in the low-frequency regime, the first order terms of \eqref{R-E1}  predominate and  hyperbolic energy methods
are thus expected to be appropriate. In particular it is suitable to work at the same level of regularity for $a\triangleq\varrho-\varrho_\infty$ and $u,$ that is $\dot B^{\frac d2-1}_{2,1}$ (the influence of the viscous term $\cA u$ is decisive, though,  as it supplies parabolic decay estimates for both $a$ and $u$ in low frequencies). To handle the  high-frequency part of the solution,
the main difficulty comes from the convection term in the density equation, as it may cause a   loss of one derivative. This is  overcome 
in \cite{D0} by  performing  an energy method on the mixed type system \eqref{R-E1} after spectral localization.

 The result of \cite{D0} has been
  extended to  Besov spaces that are not related to $L^2.$  The original proof
 of  \cite{CD} and \cite{CMZ2}  relies on a paralinearized version
of System \eqref{R-E1} combined with a Lagrangian change of variables after spectral localization.
In a recent paper \cite{H2}, B. Haspot achieved essentially the same  result  by means of   a more elementary approach
which is based on the use of Hoff's viscous effective flux \cite{Hoff} (see also   \cite{DH} for  global
results in  more general spaces, in the density dependent viscosity coefficients case).
This eventually leads to the following statement\footnote{Throughout the paper  $z^\ell$ and $z^h$
designate the low and high frequency parts of any  tempered distribution $z,$ that is
$\cF(z^\ell)\triangleq \psi\,\cF z$ and $z^h\triangleq z-z^\ell$ where $\psi$ is a suitable smooth
compactly supported function, equal to  $1$ in a neighborhood of $0.$}:
\begin{thm}\label{thm1.1} Let $d\geq2$ and  $p$ satisfying
\begin{equation}\label{eq:p}
2\leq p \leq \min(4,2d/(d-2))\quad\hbox{and, additionally, }\  p\not=4\ \hbox{ if }\ d=2.\end{equation}
Assume  that $P'(\varrho_\infty)>0$ and that \eqref{eq:pos-visc} is fulfilled.  There exists a  constant $c=c(p,d,\lambda,\mu,P,\varrho_\infty)$ such that
if    $a_0\triangleq \varrho_0-\varrho_\infty$ is in $\dot B^{\frac d{p}}_{p,1},$
 if $u_0$ is in $\dot B^{\frac d{p}-1}_{p,1}$  and if in addition   $(a_0^\ell,u_0^\ell)\in\dot B^{\frac d2-1}_{2,1}$
   with
    \begin{equation}\label{R-E4}
\cX_{p,0}\triangleq \|(a_0,u_0)\|^\ell_{\dot B^{\frac d2-1}_{2,1}}+\|a_0\|^h_{\dot B^{\frac dp}_{p,1}}
+\|u_0\|^h_{\dot B^{\frac d{p}-1}_{p,1}}\leq c\end{equation}
then \eqref{R-E1}  has a unique global-in-time  solution $(\varrho,u)$  with  $\varrho=\varrho_\infty+a$ and
$(a,u)$ in the space $X_p$ defined by\footnote{The subspace $\,\wt\cC_b(\R_+;\dot B^{s}_{q,1})\,$ of $\,\cC_b(\R_+;\dot B^{s}_{q,1})\,$
is defined in \eqref{eq:tilde}, and the norms $\|\cdot\|_{\wt L^\infty(\dot B^s_{p,1})}$
are introduced just below Definition \ref{d:besov}.}:
$$\displaylines{
(a,u)^\ell\in \wt\cC_b(\R_+;\dot B^{\frac d2-1}_{2,1})\cap  L^1(\R_+;\dot B^{\frac d2+1}_{2,1}),\quad
a^h\in \wt\cC_b(\R_+;\dot B^{\frac dp}_{p,1})\cap L^1(\R_+;\dot B^{\frac dp}_{p,1}),
\cr u^h\in  \wt\cC_b(\R_+;\dot B^{\frac dp-1}_{p,1})
\cap L^1(\R_+;\dot B^{\frac dp+1}_{p,1}).}
$$

Furthermore, we have for some constant $C=C(p,d,\lambda,\mu,P,\varrho_\infty),$
\begin{equation}\label{R-E5}
\cX_p\leq C \cX_{p,0},
\end{equation} with
\begin{multline}\label{eq:Xp}
 \cX_{p}\triangleq\|(a,u)\|^{\ell}_{\wt L^\infty(\dot B^{\frac d2-1}_{2,1})}+\|(a,u)\|^{\ell}_{L^1(\dot B^{\frac d2+1}_{2,1})}
\\+\|a\|^{h}_{\wt L^\infty(\dot B^{\frac dp}_{p,1})\cap L^1(\dot B^{\frac dp}_{p,1})}
+\|u\|^{h}_{\wt L^\infty(\dot B^{\frac dp-1}_{p,1})\cap L^1(\dot B^{\frac dp+1}_{p,1})}.
\end{multline}
\end{thm}

One may wonder how the global strong solutions constructed above  look like for large time.
Although providing an accurate  long-time asymptotic description  is still out of reach,
a number of results concerning the time decay rates of smooth global  solutions
-- sometimes referred to as $L^q-L^r$ decay estimates -- are available.
In this direction, the first achievement is due to  Matsumura and Nishida \cite{MN1,MN2} in the 80ies.
There, in the 3D case, the authors    proved the global  existence of classical solutions to  \eqref{R-E1}
supplemented with data $(\varrho_0,u_0)$  which are small perturbations in $L^1(\mathbb{R}^3)\times H^3(\mathbb{R}^3)$
of $(\varrho_\infty,0),$ and established  the following fundamental decay  estimate:
\begin{equation} \label{R-E6a}\|(\varrho-\varrho_{\infty},u)(t)\|_{L^2}\leq C\langle t\rangle^{-\frac{3}{4}}\quad\hbox{with }\
\langle t\rangle\triangleq \sqrt{1+t^2}.\end{equation}

The decay rate in \eqref{R-E6a} (which is the same as  for the heat equation with data in $L^1(\R^3)$) turns out to be
the  optimal one  for the linearized system \eqref{R-E1} about~$(\varrho_\infty,0).$
For that reason, it is  often  referred to as the optimal time-decay rate.

Shortly after Matsumura and Nishida,  still for data with high Sobolev regularity,
Ponce obtained in  \cite{P}  the following optimal $L^{p}$ decay rates for \eqref{R-E1}:
\begin{equation}
\|\nabla^{k}(\varrho-\varrho_{\infty},u)(t)\|_{L^p}\leq C\langle t\rangle^{-\frac{d}{2}(1-\frac{1}{p})-\frac k2}, \ \ \ 2\leq p\leq \infty, \ \ 0\leq k\leq 2,\ \ d=2,3. \label{R-E6}
\end{equation}
Similar results have  been established in some  situations
where the fluid domain is not $\R^d$: the half-space or exterior domain cases have  been investigated by   Kagei and Kobayashi in \cite{KK1,KK2}, Kobayashi in \cite{K},
and Shibata and Kobayashi in \cite{KS}.
\smallbreak
To find out  what kind  of asymptotic behavior is likely to
be true for  the global strong solutions of  the compressible Navier-Stokes equations constructed above, 
 it is natural to first investigate the decay  properties
 of the  linearized system \eqref{R-E1} about $(\varrho_\infty,0).$  As observed by different authors,
this is strongly connected to  the information given by  wave propagation. In that respect,
one may mention the work by  Zeng in \cite{Z} dedicated  the one-dimensional case, 
and the detailed analysis of the Green function for the multi-dimensional case carried out by Hoff and Zumbrun in  \cite{HZ1,HZ2}, 
that leads to  $L^p$ decay rates
  towards diffusion waves that are the  same  as in \eqref{R-E6}.   In \cite{LW}, Liu and Wang exhibited
   pointwise estimates of  diffusion waves with the optimal time-decay rate in odd dimension (as having Huygens's principle
 plays an important role  therein).
Let us finally mention the recent work by   Guo and Wang in \cite{GW} that
uses homogeneous Sobolev norms  of \emph{negative order} and allows  to get  optimal rates
without resorting to time-decay properties of the linear system.


\section{Main results}

Let us emphasize that all the aforementioned  works concern solutions with \emph{high Sobolev regularity.}
The optimal time-decay estimates issue for \eqref{R-E1} in the critical regularity framework has been addressed only very recently,
by Okita in \cite{O}. There, thanks to a smart modification of the method
of \cite{D0}, Inequality  \eqref{R-E6} with $k=0$ is proved in the $L^2$ critical framework in dimension $d\geq3$
provided the data are additionally in some 
superspace of $L^1$. In the survey paper \cite{D8}, the first author proposed  another description of the
time decay which allows to handle  dimension $d\geq2$ in the $L^2$ critical framework.
\smallbreak
Our aim here is to develop the method of \cite{D8} so as to establish optimal decay results \emph{in the general $L^p$ critical framework of Theorem \ref{thm1.1} and  in  any dimension $d\geq2.$} As a by-product, we shall actually
obtain a very accurate description of the decay rates, not only for Lebesgue spaces, but also
for a full family of  Besov norms with negative or positive  regularity indices.
\medbreak
Before writing out the main statement of our paper, we need to introduce  some notation and definition.
To start with, we need  a  \emph{Littlewood-Paley decomposition}.
To this end,  we fix some  smooth radial non increasing function $\chi$
supported in $B(0,\frac 43)$ and with value $1$ on $B(0,\frac34),$ then set
$\varphi(\xi)=\chi(\xi/2)-\chi(\xi)$ so that
$$
\qquad\sum_{k\in\Z}\varphi(2^{-k}\cdot)=1\ \hbox{ in }\ \R^d\setminus\{0\}
\quad\hbox{and}\quad \Supp\varphi\subset \big\{\xi\in\R^d : 3/4\leq|\xi|\leq8/3\big\}\cdotp
$$
The homogeneous dyadic blocks $\ddk$ are defined by
$$\ddk u\triangleq\varphi(2^{-k}D)u=\cF^{-1}(\varphi(2^{-k}\cdot)\cF u)=2^{kd}h(2^k\cdot)\star u
\quad\hbox{with}\quad h\triangleq\cF^{-1}\varphi.
$$
The Littlewood-Paley decomposition of a general  tempered distribution $f$ reads
\begin{equation}\label{eq:decompo}
f=\sum_{k\in\Z}\ddk f.
\end{equation}
As it holds  only modulo polynomials, it is convenient to consider only 
tempered distributions $f$ such that
\begin{equation}\label{eq:Sh}
\lim_{k\rightarrow-\infty}\|\dot S_kf\|_{L^\infty}=0,
\end{equation}
where $\dot S_kf$ stands for the low frequency cut-off defined by $\dot S_kf\triangleq\chi(2^{-k}D)f$.
Indeed, for those distributions, \eqref{eq:decompo} holds true in $\cS'(\R^d).$ 
\medbreak
Let us now turn to the definition of the main functional spaces and norms that will come into play in our paper. 
\begin{defn}\label{d:besov}
 For $s\in\R$ and
$1\leq p,r\leq\infty,$  the homogeneous Besov space $\dot B^s_{p,r}$ is  the
set of tempered distributions $f$ satisfying \eqref{eq:Sh} and
$$
\|f\|_{\dot B^s_{p,r}}\triangleq\Bigl\|\Bigl(2^{ks}\|\ddk  f\|_{L^p}\Bigr)\Bigr\|_{\ell^r(\Z)}<\infty.
$$
\end{defn}
In the case where $f$ depends also on the time variable, we shall often consider
the subspace $\wt L^\infty_T(\dot B^s_{p,1})$ of those functions of $L^\infty(0,T;\dot B^s_{p,1})$ such that
$$
\|f\|_{\wt L^\infty_T(\dot B^s_{p,1})}\triangleq\sum_{k\in\Z} 2^{ks}\sup_{t\in[0,T]}\|\ddk  f(t,\cdot)\|_{L^p}<\infty.
$$
Restricting the above norms to the low or high frequencies parts of distributions will be
fundamental in our approach. To this end,  we fix some suitable  integer $k_0$ (the value of which will follow from the proof of the main theorem)
and put\footnote{Note that for technical reasons, we need a small
overlap between low and high frequencies.} 
$$\displaylines{
\|f\|_{\dot B^s_{p,1}}^\ell\triangleq\sum_{k\leq k_0} 2^{ks}\|\ddk f\|_{L^p}\quad\hbox{and}\quad
\|f\|_{\dot B^s_{p,1}}^h\triangleq\sum_{k\geq k_0-1} 2^{ks}\|\ddk f\|_{L^p},\cr
\|f\|_{\wt L^\infty_T(\dot B^s_{p,1})}^\ell\triangleq\sum_{k\leq k_0} 2^{ks}\|\ddk f\|_{L^\infty_T(L^p)}\quad\hbox{and}\quad
\|f\|_{\wt L^\infty_T(\dot B^s_{p,1})}^h\triangleq\sum_{k\geq k_0-1} 2^{ks}\|\ddk f\|_{L^\infty_T(L^p)},}
$$
where, for any  Banach space $X,$ we denote by $L^\infty_T(X)\triangleq L^{\infty}([0,T];X)$
the  set  of  essentially bounded mesurable functions from $[0,T]$ to $X.$ \medbreak
Finally, we agree that throughout the paper $C$ stands for a positive harmless ``constant", the meaning of which is clear
 from the context. Similarly,  $f\lesssim g$ means that  $f\leq Cg$ and $f\thickapprox g$ means that $f\lesssim g$ and
$g\lesssim f$.  It will be also understood that  $\|(f,g)\|_{X}\triangleq\|f\|_{X}+\|g\|_{X}$ for all  $f,g\in X$.
\medbreak
To simplify the presentation, it is wise to perform a suitable  rescaling  so as to reduce our study to
the case where, at infinity, the density  $\varrho_\infty,$ the sound speed $c_\infty\triangleq\sqrt{P'(\varrho_\infty)}$
and the total viscosity $\nu_\infty\triangleq \lambda_\infty+2\mu_\infty$ (with
$\lambda_{\infty}\triangleq\lambda(\varrho_\infty)$ and $\mu_{\infty}\triangleq\mu(\varrho_\infty)$)
are equal to $1.$
 This may be done by making  the change of unknowns:
\begin{equation}\label{eq:change}
\wt a (t,x)\triangleq
\frac\varrho{\varrho_\infty}\biggl(\frac{\nu_\infty}{\varrho_\infty c_\infty^2}\,t,\frac{\nu_\infty}{\varrho_\infty c_\infty}\,x\biggr)-1
\quad\hbox{and}\quad \wt u(t,x)\triangleq\frac u{c_\infty}\biggl(\frac{\nu_\infty}{\varrho_\infty c_\infty^2}\,t,\frac{\nu_\infty}{\varrho_\infty c_\infty}\,x\biggr)\cdotp
\end{equation}
Assuming that\footnote{For the statement of decay estimates in
the general case, the reader is referred to Section \ref{sec:4}.}
 $\varrho_\infty=1,$ $P'(\varrho_\infty)=1$ and  $\nu_\infty=1,$ our main result is the following one.
\begin{thm}\label{thm2.1} Let $d\geq2$ and  $p$ satisfying Condition \eqref{eq:p}.
 Let  $(\varrho_0,u_0)$ fulfill the assumptions of Theorem~\ref{thm1.1}, and denote by~$(\varrho,u)$
 the corresponding global solution of System \eqref{R-E1}.
There exists a    positive constant $c=c(p,d,\lambda,\mu,P)$ such that if in addition
\begin{equation}\label{R-E9}
{\cD}_{p,0}\triangleq \|(a_0, u_0)\|_{\dot{B}^{-s_{0}}_{2,\infty}}^\ell \leq c\quad\hbox{with}\quad s_{0}\triangleq d\biggl(\frac2p-\frac12\biggr),
\end{equation}
then   we have  for all $t\geq0,$
\begin{equation}\label{R-E10}
{\cD}_{p}(t)\lesssim \bigl({\cD}_{p,0}+\|(\nabla a_0,u_0)\|^{h}_{\dot B^{\frac dp-1}_{p,1}}\bigr),
\end{equation}
where the norm ${\cD}_{p}(t)$ is defined by
\begin{multline}\label{eq:Dp}
{\cD}_{p}(t)\triangleq\sup_{s\in(-s_{0},2]}\|\langle\tau\rangle^{\frac {s_0+s}2}(a,u)\|_{L^\infty_t(\dot B^s_{2,1})}^\ell
\\+\|\langle\tau\rangle^{\alpha}(\nabla a,u)\|_{\wt L^\infty_t(\dot B^{\frac dp-1}_{p,1})}^h
+\|\tau\nabla  u\|_{\wt L^\infty_t(\dot B^{\frac dp}_{p,1})}^h
\end{multline}
and  $\alpha\triangleq\frac {s_{0}}{2}+\min(2,\frac {d}{4}+\frac12-\varepsilon)$ for some
arbitrarily small   $\varepsilon>0.$
\end{thm}
Some comments are in order.
\begin{enumerate}
\item There is some freedom in the choice  of $s_0.$
In the standard case $p=2$ and for regular solutions, it is usual
to assume that the data are in $L^1$ which, by critical embedding corresponds to $s_0=-d/2.$
 This value of $s_0$ is  relevant in other contexts like the  Boltzmann equation (see the work by Sohinger and Strain \cite{SS}), or hyperbolic systems with dissipation (see the paper by the second author and Kawashima  \cite{XK2}). 

The reason why the space $L^1$ is natural when working in  $L^2$-type framework is just because products of two terms in $L^2,$
are in $L^1.$ In our  $L^p$ framework, the similar heuristics would 
bring us to replace $L^1$ by $L^{p/2}.$
Choosing  $s_0$ as above corresponds exactly to the critical embedding $L^{p/2}\hookrightarrow \dot B^{-s_0}_{2,\infty}.$
\item The decay rate for the low frequencies of the solution (first term of $\cD_p$) is optimal
inasmuch as it corresponds to the one of the linearized system \eqref{R-E1} about $(\varrho_\infty,0)$
for general data  in $\dot B^{-s_0}_{2,\infty}.$   The last term of $\cD_p$ is consistent with the critical functional
framework given by the bulk regularity of the velocity. Finally, the (maximal) value of $\alpha$ in the second term of $\cD_p$
may be guessed from the fact that in order to close the estimates,  we need
$\|u^\ell\cdot\nabla u^h\|_{\dot B^{\frac dp-1}_{p,1}}$ to  decay like $\tau^{-\alpha},$
while  the decay of $\nabla u^h$ in $\dot B^{\frac dp}_{p,1}$ is only $\tau^{-1}.$
Then applying the following product law in Besov spaces:
$$
\| \tau^\alpha  u^\ell\cdot\nabla u^h\|_{\wt L^\infty_t(\dot B^{\frac dp-1}_{p,1})} \lesssim
\|\tau^{\alpha-1}u^\ell\|_{\wt L^\infty_t(\dot B^{\frac d2-1}_{2,1})}\| \tau\nabla u\|_{\wt L^\infty_t(\dot B^{\frac dp}_{p,1})},
$$
and using the low frequency decay rate for $u$ gives us the constraint
$\alpha-1<\frac12(s_0+\frac d2-1)$ (at least if $\frac d2-1\leq2$). \smallbreak

\item  If  replacing \eqref{R-E9} by the stronger assumption $ \|(a_0, u_0)\|_{\dot{B}^{-s_{0}}_{2,1}}^\ell \leq c,$
then one can take $\varepsilon=0$ and change the first term of $\cD_p(t)$ for the slightly stronger norm
$\sup_{s\in(-s_{0},2]}\|\langle\tau\rangle^{\frac{s+s_0}2}(a,u)\|_{\wt L^\infty_t(\dot B^s_{2,1})}^\ell.$
\smallbreak
\item In physical dimensions $d=2,3,$  Condition \eqref{eq:p} allows us to consider the case  $p>d,$ so that the  
regularity exponent $d/p-1$ for the velocity becomes negative.
Our result thus applies  to \emph{large} highly oscillating initial velocities (see  \cite{CD,CMZ2} 
for more explanation).
\smallbreak
\item  Our functional  $\cD_p$
 has been worked out to encode enough decay information to handle all the nonlinear terms.
 Having a more accurate description than  in \cite{O} and, in particular, exhibiting
 gain of regularity and decay altogether (last term of $\cD_p$) is the key to getting optimal
 decay estimates in dimension $d=2$ and for $p>2.$
 Let us also emphasize that one may deduce $L^q$-$L^{r}$ decay estimates  in the spirit of \eqref{R-E6}
 from the expression of $\cD_p$ (see Corollaries \ref{cor1.1} and \ref{cor1.2} below).\smallbreak
 \item  Last but not least,   our approach is very robust :
suitable modifications  of the definition of $\cD_p$ should allow  to prove optimal decay estimates in critical
spaces for other  hyperbolic-parabolic systems arising in fluid mechanics models.
\end{enumerate}
\medbreak
We end this section with an overview of our strategy.  The starting point is  to rewrite   System \eqref{R-E1} as  the
linearized compressible Navier-Stokes equations about $(1,0)$, looking at the nonlinearities as source terms.
More concretely, we  consider
\begin{equation}\label{R-E3}
\left\{\begin{array}{l}\d_ta+\div u=f,\\[1ex]
\d_tu-\cA u+\nabla a=g,
\end{array}\right.
\end{equation}
with $f\triangleq-\div(au),\,$
$\cA\triangleq\mu_\infty\Delta+(\lambda_\infty\!+\!\mu_\infty)\nabla\div$ such that $\mu_\infty>0$ and $\lambda_\infty+2\mu_\infty=1,$
$$  g\triangleq-u\cdot\nabla u-I(a)\cA u-k(a)\nabla a+\frac1{1+a}\div\bigl(2\wt\mu(a) D(u)+\wt\lambda(a)\div u\:\Id\bigr),
 $$
where\footnote{In our analysis, the exact value of functions $k,$ $\tilde{\lambda},$ $\tilde{\mu}$ and even $I$ will not  matter :
 we  only need those functions to be smooth enough and to vanish at $0$.}
  $$
  I(a)\triangleq\frac{a}{1+a},\quad\!\! k(a)\triangleq\frac{P'(1+a)}{1+a}-1,\quad\!\!
\wt\mu(a)\triangleq\mu(1+a)-\mu(1)\ \hbox{  and  }\  \wt\lambda(a)\triangleq\lambda(1+a)-\lambda(1).$$

In the case of high Sobolev regularity, the basic method  to prove   \eqref{R-E6a}  is to take
advantage of the corresponding $L^1-L^2$ estimates for the semi-group generated by the left-hand side of \eqref{R-E3},
treating the terms $f$ and $g$ by means of  Duhamel formula, and accepting loss of derivatives as the case may be.
  In the critical regularity framework however,  one cannot afford any loss
  of regularity  for the high frequency part of the solution (and some terms like $u\cdot\nabla a$ induce  a loss of one
derivative as one cannot expect any smoothing for $a,$  solution of a  transport equation).
As regards the well-posedness issue in  the $L^2$ critical  framework (that is Theorem
\ref{thm1.1} with $p=2$), that difficulty has been overcome in \cite{D0}  thanks to  an appropriate  energy method
 after spectral localization of the mixed hyperbolic-parabolic system \eqref{R-E3} \emph{including the convection terms}.
 As pointed out  by Okita in \cite{O}, if assuming in addition that the initial data are in $L^1$ (or rather in
 the larger Besov space $\dot B^0_{1,\infty}$) then the same arguments lead to optimal time-decay estimates
 in the $L^2$ critical framework if $d\geq3.$ 
 \smallbreak
 To prove Theorem \ref{thm2.1} in its full generality, one has to proceed differently: on one hand, using Okita's time decay functional 
 does not allow to cover the two-dimensional case, and on the other hand it is not adapted to the general $L^p$ setting.
 To get round the difficulty arising from the regularity-loss for the high-frequency part of the solution
in the $L^p$ critical framework,  we shall follow the approach that has  been used recently by Haspot \cite{H2} to prove
Theorem \ref{thm1.1}. It is based on the observation that,  at leading order, both the divergence-free part of $u$ and   
the so-called  \emph{effective velocity} $w$ (which is another name for the \emph{viscous effective flux} of D. Hoff in \cite{Hoff})
fulfill some constant coefficient  heat equation,  while $a$ satisfies a damped transport equation. 
Now, to cover all dimensions $d\geq2$ and values of $p$ satisfying \eqref{eq:p}, we also need to include 
 an additional decay information for the high frequencies of the velocity in the definition of the decay functional $\cD_p$
 (last term of \eqref{eq:Dp}). 
 
 Another difficulty if $p>2,$  compared to the $L^2$ case,  is that  one cannot expect
  interaction between high frequencies to provide any $L^1$ information 
on the low frequencies. Indeed, let us look at  the term  $u\cdot\nabla a$ as an example.
Decomposing $a$ and $u$ into their low and high frequency parts, we get
\begin{equation}
(u\cdot\nabla a)^{\ell}=(u^{\ell}\cdot\nabla a^{\ell})^{\ell}+(u^{h}\cdot\nabla a^{\ell})^{\ell}+(u^{\ell}\cdot\nabla a^{h})^{\ell}+(u^{h}\cdot\nabla a^{h})^{\ell}. \label{R-E8}
\end{equation}
As pointed out in  Theorem \ref{thm1.1},  the high-frequency part $(a^h,u^h)$  lies in an $L^p$-type space.
Hence one cannot  bound  the last term of \eqref{R-E8} in
a functional space with integrability index below $p/2.$
A similar difficulty arises for  the second and third terms which, at most,  are in $L^r$ with $1/r=1/p+1/2.$
The strategy will thus be to bound the low frequencies of $u\cdot\nabla a$  in $L^{p/2},$
and  to  resort to $L^{p/2}-L^2$ type decay estimates instead of $L^1-L^2$ estimates.  
This heuristics turns out to  work if   $p\leq d.$ Unfortunately, if  $p>d$ (a case that may occur
 in physical dimension $d=2,3$), the low frequency part of some nonlinear terms need not be   in $L^{p/2}.$
 The remedy is to perform estimates in the negative Besov space $\dot{B}^{-d(\frac2p-\frac12)}_{2,\infty},$
 which corresponds to the following critical embedding:
 $$L^{\frac{p}{2}}(\mathbb{R}^{d})\hookrightarrow \dot{B}^{0}_{\frac p2,\infty}(\mathbb{R}^{d})\hookrightarrow\dot{B}^{-d(\frac2p-\frac12)}_{2,\infty}(\mathbb{R}^{d}),$$
 and matches  the low frequency assumption \eqref{R-E9}. This requires our using some non so-classical 
 product estimates in Besov spaces where  the low frequency cut-off is crucial, see  Proposition \ref{prop2.4b}.
  \medbreak
The rest of the paper unfolds as follows.  Section~\ref{sec:3} is devoted to the
proof of Theorem \ref{thm2.1}.  In Section~\ref{sec:4}, we point out some  consequences of our main theorem :   optimal decay estimates in Lebesgue
spaces in the spirit of \eqref{R-E6a}, and explicit  dependency
 with respect to  the Mach and Reynolds number
of \eqref{R-E1} (recall that, so far, we set  all parameters of the system  to $1$ for simplicity).
Some material concerning Besov spaces, paradifferential calculus, product
and commutator estimates  is recalled in Appendix.


\section{The proof of time-decay estimates} \setcounter{equation}{0}\label{sec:3}

We here prove  Theorem \ref{thm2.1}, taking  for granted the  global existence result of  Theorem \ref{thm1.1}.
We proceed in  three  steps, according to the three terms of  the time-weighted functional  $\cD_{p}$ defined in \eqref{eq:Dp}.
\smallbreak
 Step 1  combines
  the low frequency decay properties of the semi-group  defined by  the left-hand side of \eqref{R-E3},
 and Duhamel principle to handle the nonlinear terms.
 In that step,  having Condition \eqref{R-E9} is fundamental,
 as it rules  the decay rate  of   the   low-frequency part of $\cD_p$ (and thus indirectly  of  high frequency terms,
 owing to nonlinear interaction). The proof for  $p>d$ turns out to be  more involved as for $p\leq d$
 because  the space $\dot B^{\frac dp-1}_{p,1}$ is no longer
 (locally) embedded in $L^p$ so that  one cannot  resort directly
  to the obvious product law $L^p\times L^p\to L^{\frac p2}$
 to treat nonlinearities. 
\smallbreak
In the second step,  in order to exhibit the decay of the high frequencies part of the solution,  we
 introduce (after  B. Haspot in \cite{H1,H2}) the \emph{effective velocity} $w\triangleq\nabla(-\Delta)^{-1}(a-\div u).$
 This is motivated by the observation that if \eqref{R-E3} is written in terms of $a,$ $w$ and of the divergence free part $\cP u$
 of $u,$ then, up to low order terms, $a$ satisfies a \emph{damped} transport equation, and
 both $w$ and $\cP u$ satisfy a heat equation.
Applying a suitable  energy method after spectral localization  
enables us to avoid the  loss of one derivative coming from the convection terms, 
and to take advantage of  the  nice decay properties provided by the heat and damped transport equations.

In the last step, we establish gain of regularity and decay altogether
for the high frequencies of the velocity.  That step strongly relies on the
maximal regularity estimates for the Lam\'e semi-group which are the same as that for heat semi-group (see the remark
that follows Prop. \ref{prop2.7}), and is fundamental to get our main result in any dimension $d\geq2$ and for 
all indices $p$ satisfying \eqref{eq:p}.
\bigbreak
In what follows, we shall use repeatedly that
for $0<\sigma_1\leq\sigma_2$ with $\sigma_2>1,$ we have
\begin{equation}\label{R-E16}
\int_0^t\langle t-\tau\rangle^{-\sigma_1}\langle \tau\rangle^{-\sigma_2}\,d\tau
\lesssim\langle t\rangle^{-\sigma_1}.
\end{equation}

\subsubsection*{Step 1: Bounds for the low frequencies}
Let  $(E(t))_{t\geq0}$ be the semi-group associated to the left-hand side of \eqref{R-E3}. We  get after
spectral localization\footnote{Throughout, we set $z_k\triangleq\ddk z$ for any tempered distribution $z$ and $k\in\Z.$}
 that   for all $k\in\Z,$
\begin{equation}\label{R-E17}
\left(\begin{array}{c}a_k(t)\\u_k(t)\end{array}\right)
={E(t)}\left(\begin{array}{c}a_{0,k}\\u_{0,k}\end{array}\right)
\!+\!\int_0^t E(t-\tau)\left(\begin{array}{c}f_k(\tau)\\g_k(\tau)\end{array}\right)d\tau
\end{equation}
where $f$ and $g$ have been defined in \eqref{R-E3}.
\medbreak
{From} an explicit computation of  the action of $E(t)$ in Fourier variables (see e.g. \cite{CD}), we
know that for any $k_0\in\Z,$ there exist two positive constants $c_{0}$ and $C$ depending
only on $k_0$ and such that
\begin{eqnarray}\label{R-E18}
|\cF(E(t) U)(\xi)|\leq Ce^{-c_0t|\xi|^2}|\cF U(\xi)|\quad\hbox{for all }\  |\xi|\leq2^{k_0}.
\end{eqnarray}
Therefore, using Parseval equality and the definition of $\ddk,$ we get
 for all $k\leq k_0$,
\begin{align}\label{R-E109}
\|E(t)\ddk U\|_{L^2}\leq  C e^{-\frac{c_0}{4}2^{2k}t}\|\ddk U\|_{L^2}.
\end{align}
Hence, multiplying by $t^{\frac{s_0+s}2}2^{ks}$ and summing up on $k\leq k_0,$
\begin{equation}\label{R-E110}
t^{\frac {s_0+s}2}\sum_{k\leq k_0}2^{ks}\|E(t)\ddk U\|_{L^2}\lesssim
\|U\|_{\dot{B}^{-s_{0}}_{2,\infty}}^\ell \sum_{k\leq k_0}(\sqrt t\,2^k)^{s_{0}+s}\,e^{-\frac{c_0}{4}(\sqrt t\,2^k)^2}.
\end{equation}
Due to the following fact: for any $\sigma>0$ there exists a constant $C_\sigma$ so that
\begin{equation}\label{R-E111}
\sup_{t\geq0}\sum_{k\in\Z}t^{\frac\sigma2}2^{k\sigma}e^{-\frac{c_0}42^{2k}t}\leq C_\sigma,
\end{equation}
we get from \eqref{R-E110} that for $s+s_{0}>0,$
\begin{equation}\label{R-E112}
\sup_{t\geq0}\, t^{\frac {s_0+s}2}\|E(t)U\|_{\dot B^s_{2,1}}^\ell\lesssim \|U\|_{\dot{B}^{-s_{0}}_{2,\infty}}^\ell.
\end{equation}
In addition, it is clear  that  for $s+s_{0}>0,$
\begin{equation}\label{R-E113}
\|E(t)U\|_{\dot B^s_{2,1}}^\ell\lesssim \|U\|_{\dot{B}^{-s_{0}}_{2,\infty}}^\ell\Big(\sum_{k\leq k_{0}}2^{k(s_{0}+s)}\Big) \lesssim \|U\|_{\dot{B}^{-s_{0}}_{2,\infty}}^\ell.
\end{equation}
Therefore, setting $\langle t\rangle\triangleq \sqrt{1+t^2},$ we arrive at
\begin{equation}\label{R-E114}
\sup_{t\geq0}\langle t\rangle^{\frac {s_0+s}2}\|E(t)U\|_{\dot B^s_{2,1}}^\ell\lesssim \|U\|_{\dot{B}^{-s_{0}}_{2,\infty}}^\ell,
\end{equation}
and thus, taking advantage of Duhamel's formula,
\begin{equation}\label{R-E114a}
\biggl\|\int_0^tE(t-\tau)(f(\tau),g(\tau))\,d\tau\biggr\|^\ell_{\dot B^s_{2,1}}
\lesssim\int_0^t\langle t-\tau\rangle^{-\frac{s_0+s}2}\|(f(\tau),g(\tau))\|_{\dot B^{-s_0}_{2,\infty}}^\ell\,d\tau.
\end{equation}
We claim that if $p$ fulfills \eqref{eq:p}, then we have for all $t\geq0,$
\begin{equation}\label{R-E25b}
\int_0^t\langle t-\tau\rangle^{-\frac {s_0+s}2} \|(f, g)(\tau)\|_{\dot B^{-s_0}_{2,\infty}}^\ell\,d\tau\lesssim\langle t\rangle^{-\frac {s_0+s}2}
\bigl(\cD^2_{p}(t)+\cX^2_{p}(t)\bigr),
\end{equation}
 where $\cX_{p}$ and $\cD_p$ have been defined in
\eqref{eq:Xp} and \eqref{eq:Dp}, respectively.
\medbreak
Proving our claim requires different arguments depending on whether $p\leq d$ or $p>d.$
Let us start with the easier case $2\leq p\leq d.$ Then, owing to the embedding $L^{p/2}\hookrightarrow \dot B^{-s_0}_{2,\infty},$
it suffices to establish that
\begin{equation}\label{R-E25}
\int_0^t\langle t-\tau\rangle^{-\frac {s_0+s}2} \|(f, g)(\tau)\|_{L^{p/2}}\,d\tau\lesssim\langle t\rangle^{-\frac {s_0+s}2}
\bigl(\cD^2_{p}(t)+\cX^2_{p}(t)\bigr).
\end{equation}
Now, to bound the term with $f,$ we use the  decomposition
\begin{equation}\label{R-E26}
f=u\cdot\nabla a+ a\, \div u^\ell + a\,\div u^h,
\quad\hbox{with }\ u^\ell\triangleq\sum_{k< k_0}\ddk u\ \hbox{ and }\  u^h\triangleq u-u^\ell.
\end{equation}
 It follows from H\"{o}lder inequality  that
\begin{multline}\label{R-E27}
\int_0^t\langle t-\tau\rangle^{-\frac{s_{0}+s}2}\|(u\cdot\nabla a)(\tau)\|_{L^{p/2}}\,d\tau\\
\leq  \Bigl(\sup_{0\leq\tau\leq t}\langle \tau\rangle^{\frac{s_0}{2}-\frac{d}{2p}+\frac d4}\|u(\tau)\|_{L^p}\Bigr)
\Bigl(\sup_{0\leq\tau\leq t}\langle \tau\rangle^{\frac{s_0}{2}-\frac{d}{2p}+\frac d4+\frac12}\|\nabla a(\tau)\|_{L^p}\Bigr)
\\ \times \int_0^t\langle t-\tau\rangle^{-\frac{s_{0}+s}2}\langle\tau\rangle^{-s_{0}+\frac dp-\frac {d+1}{2}}\,d\tau.
\end{multline}
By Minkowski's inequality, we have
\begin{eqnarray}\label{R-E28}
\|u\|_{L^p}\leq \|u^{\ell}\|_{L^p}+\|u^{h}\|_{L^p},
\end{eqnarray}
and  embedding  (see the Appendix)  together  with the definition of $u^\ell$ and $u^h$ imply that
\begin{equation}\label{R-E30}
\|u^{\ell}\|_{L^p}\lesssim   \|u\|^{\ell}_{\dot{B}^{\frac d2-\frac dp}_{2,1}}
\quad\hbox{and} \quad \|u^{h}\|_{L^p} \lesssim  \|u\|_{\dot{B}^{\frac{d}{p}-1}_{p,1}}^{h}\quad\hbox{if }\
 2\leq p\leq d.
\end{equation}
 Combining \eqref{R-E28}--\eqref{R-E30}  and using the definition of $\alpha$ and  of $\cD_{p}(t)$ thus yields
\begin{eqnarray}\label{R-E31}
\sup_{0\leq\tau\leq t}\langle \tau\rangle^{\frac{s_0}{2}-\frac{d}{2p}+\frac d4}\|u(\tau)\|_{L^p}\lesssim \cD_{p}(t).
\end{eqnarray}
Similarly, we can get
\begin{eqnarray}\label{R-E32}
\sup_{0\leq\tau\leq t}\langle \tau\rangle^{\frac{s_0}{2}-\frac{d}{2p}+\frac d4}\|a(\tau)\|_{L^p}+\sup_{0\leq\tau\leq t}\langle \tau\rangle^{\frac{s_0}{2}-\frac{d}{2p}+\frac d4+\frac12}\|\nabla a(\tau)\|_{L^p}\lesssim \cD_{p}(t).
\end{eqnarray}
Because  $2\leq p\leq d$  and  $s_0=\frac{2d}p-\frac d2$,   we arrive for all  $-s_0<s\leq2$ at
\begin{eqnarray}\label{R-E33}
\frac{s_0}{2}+\frac s2\leq\frac dp-\frac d4+1\leq s_0-\frac dp+\frac {d+1}{2}=\frac dp+\frac12\cdotp
\end{eqnarray}
Since $d/p+1/2>1,$  we get from \eqref{R-E16},
\begin{eqnarray}\label{R-E34}
\int_0^t\langle t-\tau\rangle^{-\frac{s_0+s}2}\langle\tau\rangle^{-s_{0}+\frac dp-\frac {d+1}{2}}\,d\tau \lesssim \langle t\rangle^{-\frac{s_0+s}2}.
\end{eqnarray}
Hence, it follows from  \eqref{R-E27}, \eqref{R-E31}, \eqref{R-E32}  that
\begin{eqnarray}\label{R-E35}
\int_0^t\langle t-\tau\rangle^{-\frac{s_0+s}2}\|(u\cdot\nabla a)(\tau)\|_{L^{p/2}}\,d\tau \lesssim \langle t\rangle^{-\frac{s_0+s}2} \cD_{p}^2(t).
\end{eqnarray}
Bounding the term with $a\,\div u^\ell$ is similar: we get
\begin{eqnarray}\label{R-E36}
\int_0^t\langle t-\tau\rangle^{-\frac{s_0+s}2}\|(a\,\div u^\ell) (\tau)\|_{L^{p/2}}\,d\tau \lesssim
\langle t\rangle^{-\frac{s_0+s}2} \cD_{p}^2(t).
\end{eqnarray}
Regarding the term with $a\,\div u^h,$ we use that if $t\geq2,$
$$\displaylines{
\int_0^t\langle t-\tau\rangle^{-\frac {s_0+s}2}\|(a\,\div u^h)(\tau)\|_{L^{p/2}}\,d\tau\hfill\cr\hfill \leq
\int_0^t\langle t-\tau\rangle^{-\frac{s_{0}+s}2} \|a(\tau)\|_{L^p}\|\div u^h(\tau)\|_{L^p}\,d\tau =
\Big(\int^{1}_{0}+\int^{t}_{1}\Big)(\cdot\cdot\cdot)d\tau\triangleq I_{1}+I_{2}.}$$
Remembering the definitions of $\cX_{p}(t)$ and $\cD_{p}(t)$, we can obtain
\begin{eqnarray}\label{R-E38}
I_{1}&\!\!\lesssim\!\!& \langle t\rangle^{-\frac{s_{0}+s}2} \Big(\sup_{0\leq \tau \leq 1}\|a(\tau)\|_{L^p}\Big)\int^{1}_{0}\|\div u^h(\tau)\|_{L^p}\,d\tau
\nonumber\\ &\!\!\lesssim\!\!& \langle t\rangle^{-\frac{s_{0}+s}2} \cD_{p}(1)\cX_{p}(1)
\end{eqnarray}
and, using the fact that $\langle \tau\rangle \thickapprox \tau$ when $\tau\geq1,$
\begin{eqnarray*}
I_{2}&=&\int_1^t\langle t-\tau\rangle^{-\frac {s_{0}+s}2}\langle \tau\rangle^{-1-\frac{s_0}{2}}
\bigl(\langle\tau\rangle^{\frac {s_0}{2}}\|a(\tau)\|_{L^p}\bigr)\bigl(\tau\|\div u^h(\tau)\|_{L^p}\bigr)\,d\tau
\nonumber\\ &\leq& \sup_{1\leq \tau \leq t}\Big(\langle\tau\rangle^{\frac {s_0}{2}}\|a(\tau)\|_{L^p}\Big)
\sup_{1\leq \tau \leq t}\Big(\tau\|\div u^h(\tau)\|_{L^p}\Big)\int_1^t\langle t-\tau\rangle^{-\frac{s_{0}+s}2}
\langle \tau\rangle^{-1-\frac{s_0}{2}}\,d\tau
\\ &\lesssim& \langle t\rangle^{-\frac{s_{0}+s}2} \cD_{p}^2(t).
\end{eqnarray*}
Therefore, for $t\geq2$, we arrive at
\begin{equation}\label{R-E40}
\int_0^t\langle t-\tau\rangle^{-\frac {s_{0}+s}2}\|(a\,\div u^h)(\tau)\|_{L^{p/2}}\,d\tau \lesssim\langle t\rangle^{-\frac{s_{0}+s}2} \Big(\cD_{p}(t)\cX_{p}(t)+\cD_{p}^2(t)\Big)\cdotp
\end{equation}
The case $t\leq2$ is obvious as $\langle t\rangle\thickapprox1$ and
$\langle t-\tau\rangle\thickapprox1$ for $0\leq\tau\leq t\leq 2,$ and 
\begin{eqnarray}\label{R-E41}
\int_0^t\|a\,\div u^h\|_{L^{p/2}}\,d\tau\leq \|a\|_{L^\infty_t(L^p)}\|\div u^h\|_{L_t^1(L^p)}\lesssim \cD_{p}(t)\cX_{p}(t).
\end{eqnarray}
Next, in order to bound the term of \eqref{R-E25} corresponding to $g,$ we use the decomposition
$g=g^1+g^2+g^3+g^4$ with  $g^1\triangleq -u\cdot\nabla u,$  $g^2\triangleq -k(a)\nabla a,$
 $$\begin{array}{rcl}
 g^3&\triangleq& 2\,\Frac{\tilde\mu(a)}{1+a}\,\div D(u)+\frac{\tilde\lambda(a)}{1+a}\,\nabla\div u-I(a)\cA u\\[2ex]
 \hbox{and}\qquad g^4&\triangleq&2\,\Frac{\tilde\mu'(a)}{1+a}D(u)\cdot\nabla a+\frac{\tilde\lambda'(a)}{1+a}\,\div u\:\nabla a.\end{array}
$$
 The terms with $g^1$ and $g^2$ may be handled as $f$ above:
 $k(a)\nabla a$ and  $u\cdot\nabla u^\ell$ may be treated as $u\cdot\nabla a,$
  and $u\cdot\nabla u^h,$ as $a\,\div u^h.$ To handle  the viscous term $g^3,$ we see that it suffices to bound
\begin{eqnarray}\label{R-E42}
K(a)\nabla^2 u^\ell\quad\hbox{and}\quad K(a)\nabla^2u^h,
\end{eqnarray}
where $K$ stands for some smooth function vanishing at $0.$
\medbreak
To bound  $K(a)\nabla^2u^\ell$, we write that
\begin{eqnarray}\label{R-E43}
 \int_0^t\langle t-\tau\rangle^{-\frac {s_{0}+s}2}\|K(a)\nabla^2u^\ell\|_{L^{p/2}}\,d\tau
 &\!\!\!\lesssim\!\!\!& \int_0^t \langle t-\tau\rangle^{-\frac {s_{0}+s}2}\langle\tau\rangle^{-1-s_{0}}\,d\tau\nonumber \\
 \times\Bigl(\sup_{\tau\in[0,t]} &&\!\!\!\!\!\!\!\!\!\!\langle\tau\rangle^{\frac {s_0}{2}}\|a(\tau)\|_{L^p}\Bigr)
  \Bigl(\sup_{\tau\in[0,t]} \langle\tau\rangle^{\frac {s_0}{2}+1}\|\nabla^2 u^\ell(\tau)\|_{L^p}\Bigr)\nonumber \\&
 \!\!\!\lesssim\!\!\!& \langle t\rangle^{-\frac{s_{0}+s}2} \cD_{p}^2(t),
\end{eqnarray}
where we used \eqref{R-E32}, the fact that $\frac d4-\frac d{2p}\geq0,$ and 
$$
\|\nabla^2 u^\ell(\tau)\|_{L^p}\lesssim \|u\|^\ell_{\dot B^{\frac d2-\frac dp+2}_{2,1}}
\lesssim\|u\|^\ell_{\dot B^2_{2,1}}.
$$

To bound the term involving $K(a)\nabla^2u^h,$ we consider the cases $t\geq2$ and $t\leq2$ separately.  If  $t\geq2$ then we write
$$\begin{aligned}
\int_0^t\langle t-\tau\rangle^{-\frac {s_0+s}2}\|K(a)\nabla^2u^h\|_{L^{p/2}}\,d\tau  \lesssim& \int_0^t\langle t-\tau\rangle^{-\frac {s_0+s}2} \|a(\tau)\|_{L^p}\|\nabla^2u^h(\tau)\|_{L^p}\,d\tau\\ =&
\Big(\int^{1}_{0}+\int^{t}_{1}\Big)(\cdot\cdot\cdot)d\tau\triangleq J_{1}+J_{2}.
\end{aligned}$$
Now, because $d/p-1\geq0,$ we have by embedding,
$$
J_{1}\lesssim \langle t \rangle ^{-\frac {s_0+s}2}\Big(\sup_{0\leq \tau \leq 1}\|a(\tau)\|_{L^p}\Big)\int_0^1\|\nabla^2u^h\|_{L^p}\,d\tau\lesssim \langle t\rangle^{-\frac{s_{0}+s}2} \cD_{p}(1)\cX_{p}(1)$$
and$$
\begin{aligned}
J_{2}&\lesssim \sup_{0\leq \tau \leq t}\Big(\langle\tau\rangle^{\frac {s_0}{2}}\|a(\tau)\|_{L^p}\Big)
\sup_{0\leq \tau \leq t}\Big(\tau\|\nabla^2 u^h(\tau)\|_{L^p}\Big)\int_1^t\langle t-\tau\rangle^{-\frac {s_{0}+s}2}\langle \tau\rangle^{-1-\frac{s_0}{2}}\,d\tau\\ &\lesssim\langle t\rangle^{-\frac{s_{0}+s}2} \cD_{p}^2(t).
\end{aligned}$$
Therefore, we end up with
$$
\int_0^t\langle t-\tau\rangle^{-\frac {s_{0}+s}2}\|K(a)\nabla^2u^h\|_{L^{p/2}}\,d\tau \lesssim \langle t\rangle^{-\frac{s_{0}+s}2} \Big( \cD_{p}(t)\cX_{p}(t)+\cD_{p}^2(t)\Big)\quad\hbox{for all }\ t\geq2.
$$
In the  case  $t\leq2$, we have $\langle t\rangle\thickapprox1$ and $\langle t-\tau\rangle\thickapprox1$ for $0\leq\tau\leq t\leq 2,$ and
\begin{equation}\label{R-E48}
  \int_0^t\langle t-\tau\rangle^{-\frac {s_{0}+s}2}\|K(a)\nabla^2u^h\|_{L^{p/2}}\,d\tau
 \lesssim   \int_0^t \|a\|_{L^p}\|\nabla^2u^h\|_{L^p}\,d\tau
\lesssim \cD_{p}(t)\cX_{p}(t).\end{equation}
To bound $g^4,$ it suffices to consider $\nabla F(a)\otimes\nabla u$ where $F$ is some
smooth function vanishing at $0.$ Once again, it is convenient to split that term into
$$\nabla F(a)\otimes\nabla u=\nabla F(a)\otimes\nabla u^\ell+\nabla F(a)\otimes\nabla u^h.$$
For bounding the first term, one may proceed as for proving \eqref{R-E43} (using \eqref{R-E32}):
$$
 \begin{array}{rcl}
 \Int_0^t\langle t-\tau\rangle^{-\frac {s_0+s}2}\|\nabla F(a)\otimes\nabla u^\ell\|_{L^{p/2}}\,d\tau
 &\!\!\!\lesssim\!\!\!& \Int_0^t \langle t-\tau\rangle^{-\frac{s_{0}+s}2}\langle\tau\rangle^{-1-s_{0}}\,d\tau\hspace{2cm}\nonumber \\
 \times\Bigl(\Sup_{\tau\in[0,t]} \langle\tau\rangle^{\frac {s_0}{2}+\frac12}&&\!\!\!\!\!\!\!\!\!\!\|\nabla a(\tau)\|_{L^p}\Bigr)
  \Bigl(\Sup_{\tau\in[0,t]} \langle\tau\rangle^{s_0+\frac12}\|\nabla u^\ell(\tau)\|_{L^p}\Bigr)\nonumber \\&
 \!\!\!\lesssim\!\!\!& \langle t\rangle^{-\frac{s_0+s}2} \cD_{p}^2(t).
\end{array}
$$

To handle the term $\nabla F(a)\otimes\nabla u^h,$ we consider the cases $t\geq2$ and $t\leq2$ separately.  If  $t\geq2$ then we write
$$\begin{aligned}
\int_0^t\langle t-\tau\rangle^{-\frac {s_{0}+s}2}\|\nabla F(a)\otimes\nabla u^h\|_{L^{p/2}}\,d\tau  \lesssim& \int_0^t\langle t-\tau\rangle^{-\frac {s_{0}+s}2} \|\nabla a(\tau)\|_{L^p}\|\nabla u^h(\tau)\|_{L^p}\,d\tau\\ =&
\Big(\int^{1}_{0}+\int^{t}_{1}\Big)(\cdot\cdot\cdot)d\tau\triangleq K_{1}+K_{2}.
\end{aligned}$$
It is clear that
$$
K_{1}\lesssim \langle t \rangle ^{-\frac{s_{0}+s}2}\Big(\sup_{0\leq \tau \leq 1}\|\nabla a(\tau)\|_{L^p}\Big)\int_0^1\|\nabla u^h\|_{L^p}\,d\tau\lesssim \langle t\rangle^{-\frac{s_{0}+s}2} \cD_{p}(1)\cX_{p}(1)$$
and that
$$K_{2}\lesssim \sup_{1\leq \tau \leq t}\Big(\langle\tau\rangle^{\frac {s_0}{2}+\frac12}\|\nabla a(\tau)\|_{L^p}\Big)
\sup_{1\leq \tau \leq t}\Big(\tau^{\frac12} \|\nabla u^h(\tau)\|_{L^p}\Big)\int_1^t\langle t-\tau\rangle^{-\frac {s_{0}+s}2}\langle \tau\rangle^{-1-\frac{s_0}{2}}\,d\tau.$$
Note that for $\tau\geq1,$ we have 
$$
\tau^{\frac12} \|\nabla u^h(\tau)\|_{L^p}\lesssim \tau \|\nabla u(\tau)\|^h_{\dot B^{\frac dp}_{p,1}}.
$$
Hence one can conclude thanks to \eqref{R-E32} that
$$
K_2\lesssim\langle t\rangle^{-\frac{s_{0}+s}2} \cD_{p}^2(t)\quad\hbox{for all }\ t\geq2.
$$
The (easy) case $t\leq2$ is left to the reader, which completes the proof of \eqref{R-E25}, and thus of \eqref{R-E25b},
for $p\leq d.$
\medbreak
Let us now prove \eqref{R-E25b} in the case $p>d$ (that can occur only if $d=2,3$).
The idea is to  replace the  H\"older inequality
$\|FG\|_{L^{p/2}}\leq\|F\|_{L^p}\|G\|_{L^p}$ with  the following two inequalities:
 \begin{eqnarray}\label{R-E115a}
\|FG\|_{\dot B^{-s_0}_{2,\infty}}
\lesssim \|F\|_{\dot B^{1-\frac dp}_{p,1}}\|G\|_{\dot B^{\frac d2-1}_{2,1}},\\
\label{R-E115b}
\|FG\|_{\dot B^{-\frac dp}_{2,\infty}}^\ell
\lesssim \|F\|_{\dot B^{\frac dp-1}_{p,1}}\|G\|_{\dot B^{1-\frac dp}_{2,1}},
 \end{eqnarray}
which stem from Proposition \ref{prop2.4} (second item) and Besov embedding.
\medbreak
Using Inequality \eqref{R-E115a} turns out to be appropriate for handling the terms:
$$
u\cdot\nabla a^\ell,\quad a\div u^\ell, \quad u\cdot\nabla u^\ell,\quad
k(a)\nabla a^\ell\ \hbox{ and }\quad K(a)\nabla^2u^\ell,
$$
while \eqref{R-E115b} will be used for $\nabla(F(a))\otimes\nabla u^\ell.$
\medbreak
We claim that
\begin{equation}\label{R-E116}
\|(a,u)(\tau)\|_{\dot B^{1-\frac dp}_{p,1}}\lesssim\langle\tau\rangle^{-\frac12}{\cD}_p(\tau).
\end{equation}
and that
\begin{equation}\label{R-E117}
\|a(\tau)\|_{\dot B^{\frac dp}_{p,1}}\lesssim\langle\tau\rangle^{-\frac dp}{\cD}_p(\tau).
\end{equation}

Indeed, we have by embedding and definition of $s_0,$
\begin{equation}\label{R-E118}
\|(a,u)^\ell(\tau)\|_{\dot B^{1-\frac dp}_{p,1}}\lesssim\|(a,u)^\ell(\tau)\|_{\dot B^{1-s_0}_{2,1}}\lesssim\langle\tau\rangle^{-\frac12}{\cD}_p(\tau),
\end{equation}
and, by interpolation, because $p>d,$
\begin{equation}\label{R-E119}
\|u^h\|_{\dot B^{1-\frac dp}_{p,1}}\lesssim\|u^h\|_{\dot B^{\frac dp-1}_{p,1}}^{\frac dp}
\|u^h\|_{\dot B^{\frac dp+1}_{p,1}}^{1-\frac dp}.
\end{equation}
Hence, using the definition of the second and third terms of ${\cD}_p,$
\begin{equation}\label{R-E120}
\|u^h(\tau)\|_{\dot B^{1-\frac dp}_{p,1}}\lesssim\langle\tau\rangle^{-(1+\frac dp(\alpha-1))}\cD_p(\tau).
\end{equation}
As obviously $\alpha\geq1,$ putting \eqref{R-E118} and \eqref{R-E120} together yields \eqref{R-E116} for $u.$
Finally, because $p\leq2d,$
$$
\|a^h(\tau)\|_{\dot B^{1-\frac dp}_{p,1}}\lesssim \|a^h(\tau)\|_{\dot B^{\frac dp}_{p,1}}
\lesssim \langle\tau\rangle^{-\alpha}{\cD}_p(\tau),
 $$
 and  Inequality \eqref{R-E116} is thus also fulfilled by $a^h.$
\smallbreak
For proving \eqref{R-E117}, we notice that by embedding and because $\frac d2\leq 2$ in the case we are interested in,
$$
\|a^\ell(\tau)\|_{\dot B^{\frac dp}_{p,1}}\lesssim \|a^\ell(\tau)\|_{\dot B^{\frac d2}_{2,1}}
\lesssim \langle\tau\rangle^{-\frac dp}\cD_p(\tau),
$$
and, because $\alpha\geq1\geq\frac dp,$
$$
\|a^h(\tau)\|_{\dot B^{\frac dp}_{p,1}}\lesssim
 \langle\tau\rangle^{-\frac dp}\cD_p(\tau).
 $$
  Now, thanks to \eqref{R-E115a} and \eqref{R-E116}, one can thus write
 $$
 \begin{aligned}
 \int_0^t\langle t-\tau\rangle^{-\frac{s_0+s}2}\|(u\cdot\nabla a^\ell)(\tau)\|_{\dot B^{-s_0}_{2,\infty}}^\ell\,d\tau
 &\lesssim \int_0^t \langle t-\tau\rangle^{-\frac{s_0+s}2}\|u\|_{\dot B^{1-\frac dp}_{p,1}}\|\nabla a^\ell\|_{\dot B^{\frac d2-1}_{2,1}}\,d\tau
 \\
  &\lesssim \cD_p^2(t)\int_0^t \langle t-\tau\rangle^{-\frac{s_0+s}2}\langle \tau\rangle^{-(\frac dp+\frac12)}\,d\tau.
  \end{aligned}
$$
For all  $-s_0<s\leq2,$ we have
\begin{equation}\label{R-E120a}
0<\frac{s_0+s}2\leq \frac dp-\frac d4+1\leq \frac dp+\frac12\cdotp
\end{equation}
 As  $d/p+1/2>1$ (because $p<2d$),   Inequality \eqref{R-E16} thus implies that
 \begin{equation}\label{R-E121}
  \int_0^t\langle t-\tau\rangle^{-\frac{s_0+s}2}\|(u\cdot\nabla a^\ell)(\tau)\|_{\dot B^{-s_0}_{2,\infty}}^\ell\,d\tau
 \lesssim \langle t\rangle^{-\frac{s_0+s}2}{\cD}_p^2(t).
 \end{equation}
The terms $a\,\div u^\ell,$ $u\cdot\nabla u^\ell$ and  $k(a)\nabla a^\ell$ are similar. 
Regarding the term $K(a)\nabla^2u^\ell,$ we just have to write that, thanks to \eqref{R-E115a} and Bernstein inequality,
$$\begin{aligned}
\int_{0}^{t}&\langle t-\tau\rangle^{-\frac {s_{0}+s}{2}}\|K(a)\nabla^2 u^\ell(\tau)\|^{\ell}_{\dot{B}^{-s_{0}}_{2,\infty}}\,d\tau\\
&\lesssim\Big(\sup_{\tau\in[0,t]}\langle\tau\rangle^{\frac{1}{2}}\|a(\tau)\|_{\dot{B}^{1-\frac{d}{p}}_{p,1}}\Big)\!
\Big(\sup_{\tau\in[0,t]}\langle\tau\rangle^{\frac{d}{p}}\|\nabla^2 u^\ell(\tau)\|_{\dot{B}^{\frac{d}{2}-1}_{2,1}}\Big)\!
\int_{0}^{t}\langle t-\tau\rangle^{-\frac {s_{0}+s}{2}}\langle\tau\rangle^{-(\frac{d}{p}+\frac{1}{2})}d\tau\\
&\lesssim \langle t\rangle^{-\frac {s_{0}+s}{2}}\Big(\sup_{\tau\in[0,t]}\langle\tau\rangle^{\frac{1}{2}}\|a(\tau)\|_{\dot{B}^{1-\frac{d}{p}}_{p,1}}\Big)\!
\Big(\sup_{\tau\in[0,t]}\langle\tau\rangle^{\frac{d}{p}}\|u^\ell(\tau)\|_{\dot{B}^{\frac{d}{2}}_{2,1}}\Big)\lesssim \langle t\rangle^{-\frac {s_{0}+s}{2}}{\cD}_{p}^2(t).\end{aligned}
$$
Bounding  $\nabla u^\ell\otimes\nabla F(a)$ requires inequality \eqref{R-E115b}
and Proposition \ref{prop2.6}: we have
$$\begin{aligned}
\int_{0}^{t}\langle t-\tau\rangle^{-\frac {s_{0}+s}{2}}\|\nabla u^\ell\cdot\nabla F(a)\|^{\ell}_{\dot{B}^{-\frac dp}_{2,\infty}}\,d\tau
&\lesssim \int_{0}^{t}\langle t-\tau\rangle^{-\frac {s_{0}+s}{2}}\|\nabla u^\ell\|_{\dot B^{1-\frac dp}_{2,1}}
\|\nabla F(a)\|_{\dot B^{\frac dp-1}_{p,1}}\,d\tau\\
&\lesssim  \int_{0}^{t}\langle t-\tau\rangle^{-\frac {s_{0}+s}{2}}\|u^\ell\|_{\dot B^{2-\frac dp}_{2,1}}
\|a\|_{\dot B^{\frac dp}_{p,1}}\,d\tau.
\end{aligned}
$$
Of course, as  $s_0\leq d/p,$ we have
$$
\|\nabla u^\ell\cdot\nabla F(a)\|^{\ell}_{\dot{B}^{-s_0}_{2,\infty}}\lesssim \|\nabla u^\ell\cdot\nabla F(a)\|^{\ell}_{\dot{B}^{-\frac dp}_{2,\infty}}.
$$
Now, because the definition of $\cD_p$ ensures that
$$
\|u^\ell\|_{\dot B^{2-\frac dp}_{2,1}}\leq \langle\tau\rangle^{-(1+\frac d{2p}-\frac d4)}\cD_p(\tau),
$$
we get, using  \eqref{R-E117} and  the fact that   $1+\frac{3d}{2p}-\frac d4\geq 1+\frac{s_0}2,$
$$
\int_{0}^{t}\langle t-\tau\rangle^{-\frac {s_{0}+s}{2}}\|\nabla u^\ell\cdot\nabla F(a)\|^{\ell}_{\dot{B}^{-s_{0}}_{2,\infty}}\,d\tau
\lesssim \langle t\rangle^{-\frac {s_{0}+s}{2}}{\cD}_{p}^2(t).
$$
Bounding the terms corresponding to
$$u\cdot\nabla a^h,\quad a\div u^h, \quad u\cdot\nabla u^h,\quad
k(a)\nabla a^h\ \hbox{ and }\quad K(a)\nabla^2u^h$$
requires our using Inequality \eqref{eq:est1} with $\sigma=1-d/p,$ namely
$$
\|FG^h\|_{\dot B^{-s_0}_{2,\infty}}^\ell\lesssim\bigl(\|F\|_{\dot B^{1-\frac dp}_{p,1}}+\|\dot S_{k_0+N_0}F\|_{L^{p^*}}\bigr)\|G^h\|_{\dot B^{\frac dp-1}_{p,1}}\quad\hbox{with }\ \frac1{p^*}\triangleq\frac 12-\frac1p,
$$
for some universal integer $N_0,$ 
which implies, owing to the embedding $\dot B^{\frac dp}_{2,1}\hookrightarrow L^{p^*}$ and to Bernstein inequality
(note that $p^*\geq p$),
\begin{equation}\label{s0}
\|FG^h\|_{\dot B^{-s_0}_{2,\infty}}^\ell\lesssim\bigl(\|F^\ell\|_{\dot B^{\frac dp}_{2,1}}+\|F\|_{\dot B^{1-\frac dp}_{p,1}}\bigr)\|G^h\|_{\dot B^{\frac dp-1}_{p,1}}.
\end{equation}

As an example, let us show how \eqref{s0} allows to bound the term corresponding to  $a\,\div u^h.$  We start with the inequality
$$\begin{aligned}
\int_{0}^{t}\langle t-\tau\rangle^{-\frac {s_{0}+s}{2}}\|a\,\div u^{h}\|^{\ell}_{\dot{B}^{-s_{0}}_{2,\infty}}\,d\tau &\lesssim\! \int_{0}^{t}\langle t-\tau\rangle^{-\frac {s_{0}+s}{2}}\bigl(\|a^\ell\|_{\dot B^{\frac dp}_{2,1}}\!+\!\|a\|_{\dot{B}^{1-\frac{d}{p}}_{p,1}}\bigr)\|\mathrm{div}u^{h}\|_{\dot{B}^{\frac{d}{p}-1}_{p,1}}d\tau \\&=
\Big(\int^{1}_{0}+\int^{t}_{1}\Big)(\cdot\cdot\cdot)d\tau\triangleq \tilde{I}_{1}+\tilde{I}_{2}.
\end{aligned}
$$
Because $d/p-1<1-d/p\leq d/p$ and $d/p\geq d/2-1,$ we have
$$
\|a\|_{\dot{B}^{1-\frac{d}{p}}_{p,1}}\lesssim \|a\|_{\dot B^{\frac dp}_{p,1}\cap \dot B^{\frac dp-1}_{p,1}}
\lesssim \|a\|_{\dot B^{\frac d2-1}_{2,1}}^\ell+\|a\|_{\dot B^{\frac dp}_{p,1}}^h\quad\hbox{and}\quad
\|a^\ell\|_{\dot B^{\frac dp}_{2,1}}\lesssim\|a\|_{\dot B^{\frac d2-1}_{2,1}}^\ell,
$$
and thus
$$\tilde{I}_{1}\lesssim \langle t\rangle^{-\frac {s_{0}+s}{2}}{\cD}_{p}(1)\cX_{p}(1).$$
Furthermore,  we notice that
\begin{equation}\label{s1}
\|(a,u)^\ell\|_{\dot B^{\frac dp}_{2,1}}\leq\langle\tau\rangle^{-\frac12(s_0+\frac dp)}\cD_p(t).
\end{equation}
Hence, using also \eqref{R-E116}, the fact that
$$
\sup_{\tau\in[1,t]}\tau\|\mathrm{div}u^{h}(\tau)\|_{\dot{B}^{\frac{d}{p}-1}_{p,1}}\lesssim \cD_p(t),
$$
and that for  all $s\leq2,$  we have $\frac{s_0+s}2\leq\min(\frac32,\frac12(s_0+\frac dp+2)),$  we conclude that
$$
\begin{aligned}\tilde{I}_{2}&\lesssim \cD_p^2(t)
\int_{0}^{t}\langle t-\tau\rangle^{-\frac {s_{0}+s}{2}}\bigl(\langle\tau\rangle^{-\frac12(s_0+\frac dp+2)}
+\langle\tau\rangle^{-\frac{3}{2}}\bigr)\,d\tau\\
&\lesssim \langle t\rangle^{-\frac {s_{0}+s}{2}}{\cD}_{p}^2(t),
\end{aligned}
$$
The term $u\cdot\nabla u^h$ being completely similar (thanks to \eqref{R-E116} and \eqref{s1}), we get
\begin{equation}\label{s6}\int_{0}^{t}\langle t-\tau\rangle^{-\frac {s_{0}+s}{2}}
\|(a\,\mathrm{div}u^{h},u\cdot\nabla u^{h})\|^{\ell}_{\dot{B}^{-s_{0}}_{2,\infty}}\,d\tau
\lesssim\langle t\rangle^{-\frac {s_{0}+s}{2}}\Big({\cD}_{p}(t)\cX_{p}(t)+{\cD}_{p}^2(t)\Big)\cdotp
\end{equation}
That \eqref{s6} also holds for  $t\leq2$  is just a consequence of  the definition
of $\cD_p$ and $\cX_p.$
\medbreak
Bounding the term with $u\cdot\nabla a^h$ works the same since $\alpha\geq1$ and
\begin{equation}
\|a^h(t)\|_{\dot B^{\frac dp}_{p,1}}\lesssim\langle\tau\rangle^{-\alpha}\cD_p(t).
\end{equation}
In order to handle the term with $k(a)\nabla a^h,$ we use \eqref{s0}, Proposition \ref{prop2.6} and proceed as follows:
$$
\begin{aligned}
\|k(a)\nabla a^h\|^{\ell}_{\dot{B}^{-s_{0}}_{2,\infty}}&\lesssim \bigl(\|k(a)\|_{\dot B^{1-\frac dp}_{p,1}}+\|\dot S_{k_0+N_0}k(a)\|_{L^{p^*}}\bigr)
\|\nabla a^h\|_{\dot B^{\frac dp-1}_{p,1}}\\
&\lesssim\bigl(\|a\|_{\dot B^{1-\frac dp}_{p,1}}+\|a\|_{L^{p^*}}\bigr)
\|\nabla a^h\|_{\dot B^{\frac dp-1}_{p,1}}.
\end{aligned}
$$
Using the embeddings $\dot B^{\frac dp}_{2,1}\hookrightarrow L^{p^*}$ and
 $\dot B^{s_0}_{p,1}\hookrightarrow L^{p^*},$ and decomposing $a$ into low and high frequencies,
 we discover that
 $$
 \|a\|_{L^{p^*}}\lesssim \|a\|^\ell_{\dot B^{\frac dp}_{2,1}}+\|a\|^h_{\dot B^{\frac dp}_{p,1}}.
 $$
 Hence
 $$
 \|k(a)\nabla a^h\|^{\ell}_{\dot{B}^{-s_{0}}_{2,\infty}}\lesssim
 \bigl(\|a\|^\ell_{\dot B^{\frac dp}_{2,1}}+\|a\|^h_{\dot B^{\frac dp}_{p,1}}\bigr)\|\nabla a^h\|_{\dot B^{\frac dp-1}_{p,1}},
 $$
 and one can thus bound the term corresponding to $k(a)\nabla a^h$ exactly as $u\cdot\nabla a^h.$
 \medbreak
 Likewise, according to \eqref{s0} and proposition \ref{prop2.6} and arguing as above, we get
 $$
 \|K(a)\nabla^2 u^h\|^{\ell}_{\dot{B}^{-s_{0}}_{2,\infty}}\lesssim
 \bigl(\|a\|^\ell_{\dot B^{\frac d2-1}_{2,1}}+\|a\|^h_{\dot B^{\frac dp}_{p,1}}\bigr)\|\nabla^2u^h\|_{\dot B^{\frac dp-1}_{p,1}}.
 $$
 As we have for all $t\geq0,$
 $$
 \int_0^t  \|\nabla^2 u\|^h_{\dot B^{\frac dp-1}_{p,1}}\,d\tau\lesssim\cX_p(t)\quad\hbox{and}\quad
 t\|\nabla^2 u^h(t)\|_{\dot B^{\frac dp-1}_{p,1}}\lesssim  \cD_p(t),
 $$
 one can conclude exactly as for the previous term $k(a)\nabla a^h$ that 
 $$
 \int_{0}^{t}\langle t-\tau\rangle^{-\frac {s_{0}+s}{2}} \|K(a)\nabla^2 u^h\|^{\ell}_{\dot{B}^{-s_{0}}_{2,\infty}}\,d\tau\lesssim
 \langle t\rangle^{-\frac {s_{0}+s}{2}}\Big({\cD}_{p}(t)\cX_{p}(t)+{\cD}_{p}^2(t)\Big)\cdotp
 $$

  Finally, to bound $\nabla F(a)\otimes\nabla u^h,$ we have to resort to \eqref{eq:est2} with $\sigma=1-d/p,$ namely
  $$
  \|\nabla F(a)\otimes\nabla u^h\|^{\ell}_{\dot{B}^{-s_{0}}_{2,\infty}}\lesssim
  \biggl(\|\nabla F(a)\|_{\dot B^{\frac dp-1}_{p,1}}+\sum_{k=k_0}^{k_0+N_0-1}\|\ddk\nabla F(a)\|_{L^{p^*}}\biggr)
  \|\nabla u^h\|_{\dot B^{1-\frac dp}_{p,1}}.
  $$
  As $p^*\geq p,$  Bernstein inequality ensures that $\|\ddk\nabla F(a)\|_{L^{p^*}}\lesssim \|\ddk F(a)\|_{L^p}$ for $k_0\leq k< k_0+N_0.$
  Hence, thanks to Proposition \ref{prop2.6}, we have
    $$
  \|\nabla F(a)\otimes\nabla u^h\|^{\ell}_{\dot{B}^{-s_{0}}_{2,\infty}}\lesssim
  \|a\|_{\dot B^{\frac dp}_{p,1}}     \|\nabla u^h\|_{\dot B^{1-\frac dp}_{p,1}}.
  $$
Therefore, if $t\geq2$ then
 $$\begin{aligned}
\int_{0}^{t}\langle t-\tau\rangle^{-\frac {s_{0}+s}{2}}\|\nabla F(a)\otimes\nabla u^h\|^{\ell}_{\dot{B}^{-s_{0}}_{2,\infty}}\,d\tau &\lesssim\int_{0}^{t}\langle t-\tau\rangle^{-\frac {s_{0}+s}{2}}\|a\|_{\dot{B}^{\frac{d}{p}}_{p,1}}\|\nabla u^h\|_{\dot{B}^{1-\frac{d}{p}}_{p,1}}\,d\tau\\&=
\Big(\int^{1}_{0}+\int^{t}_{1}\Big)(\cdot\cdot\cdot)d\tau\triangleq \tilde{K}_{1}+\tilde{K}_{2}.
\end{aligned}
$$ As $1-d/p\leq d/p,$ it is clear that
$$\tilde{K}_{1}\lesssim \langle t\rangle^{-\frac {s_{0}+s}{2}}{\cD}_{p}(1)\cX_{p}(1)$$
and that, owing to \eqref{R-E117},
$$\begin{aligned}
\tilde{K}_{2}&\lesssim \Big(\sup_{\tau\in[1,t]}\langle\tau\rangle^{\frac{d}{p}}\|a(\tau)\|_{\dot{B}^{\frac{d}{p}}_{p,1}}\Big)
\Big(\sup_{\tau\in[1,t]}\tau\|\nabla u^h(\tau)\|_{\dot{B}^{\frac{d}{p}}_{p,1}}\Big)
\int_{0}^{t}\langle t-\tau\rangle^{-\frac {s_{0}+s}{2}}\langle\tau\rangle^{-1-\frac dp}d\tau\\&\lesssim  \langle t\rangle^{-\frac {s_{0}+s}{2}}{\cD}_{p}^2(t).\end{aligned}$$
Finally, if $t\leq2$ then  we have
$$\int_{0}^{t}\langle t-\tau\rangle^{-\frac {s_{0}+s}{2}}\|\nabla F(a)\otimes\nabla u^h(\tau)\|^{\ell}_{\dot{B}^{-s_{0}}_{2,\infty}}d\tau \lesssim \|a\|_{L^\infty_t(\dot{B}^{\frac{d}{p}}_{p,1})}\|\nabla u^h\|_{L_t^1(\dot{B}^{\frac{d}{p}}_{p,1})}\lesssim {\cD}_{p}(t)\cX_{p}(t),
$$
which completes the proof of \eqref{R-E25} in the case $p>d.$
\medbreak
Combining with \eqref{R-E114} for bounding the term of \eqref{R-E27} pertaining to the data,
we conclude that
\begin{eqnarray}\label{s17}
\langle t\rangle^{\frac {s_{0}+s}2}\|(a,u)(t)\|_{\dot B^s_{2,1}}^\ell \lesssim
{\cD}_{p,0}+{\cD}^2_{p}(t)+\cX^2_{p}(t)\quad\hbox{for all  }\ t\geq0,
\end{eqnarray}
 provided that $-s_{0}<s\leq2.$


\subsubsection*{Step 2: Decay estimates for the high frequencies of $(\nabla a, u)$}
This step is devoted to bounding the second term of $\cD_{p}(t)$. In contrast with the first step, here
one can provide a common proof for all values of $p$ fulfilling \eqref{eq:p}.

Let $\cP\triangleq\Id+\nabla(-\Delta)^{-1}\mathrm{div}$ be the Leray projector onto divergence-free vector fields. It follows from \eqref{R-E3} that  $\cP u$ satisfies the following ordinary heat equation:
$$
\d_t\cP u -\mu_\infty\Delta\cP u=\cP g.
$$
Applying $\ddk$ to the above equation yields for all $k\in\Z,$
$$
\d_t\cP u_k -\mu_\infty\Delta\cP u_k=\cP g_k\quad\hbox{with }\ u_k\triangleq \ddk u\ \hbox{ and }\
g_k\triangleq \ddk g.
$$
Then, multiplying each component of the above equation
by $|(\cP u_k)^i|^{p-2}(\cP u_k)^i$ and  integrating over $\R^d$ gives  for $i=1,\cdots,d,$
$$\frac 1p\frac d{dt}\|\cP u_k^i\|_{L^p}^p-\mu_\infty\int \Delta(\cP u_k)^i|(\cP u_k)^i|^{p-2}(\cP u_k)^i\,dx
=\int|(\cP u_k)^i|^{p-2}(\cP u_k)^ig^i_k\,dx.$$
The key observation is that the second term of the l.h.s., although not spectrally localized,
may be bounded from below as if it were (see Prop. \ref{prop2.3bis}).
After summation  on $i=1,\cdots,d,$ we end up  for some constant $c_p$ with
\begin{equation}\label{R-E52}
\frac1p\frac d{dt}\|\cP u_k\|_{L^p}^p+c_p\mu_\infty2^{2k}\|\cP u_k\|_{L^p}^p
\leq \|\cP g_k\|_{L^p}\|\cP u_k\|_{L^p}^{p-1}.
\end{equation}
At this point, following Haspot's method in   \cite{H1,H2}, we introduce the \emph{effective velocity}
$$w=\nabla(-\Delta)^{-1}(a-\mathrm{div}\,u).$$
It is clear  that $w$ fulfills
\begin{eqnarray}\label{R-E53}
\d_tw-\Delta w=\nabla(-\Delta)^{-1}(f-\div g) +w-(-\Delta)^{-1}\nabla a.
\end{eqnarray}
Hence, arguing exactly as for proving \eqref{R-E52}, we get for $w_k\triangleq\ddk w$:
\begin{multline}\label{R-E54}
\frac1p\frac d{dt}\|w_k\|_{L^p}^p+c_p2^{2k}\|w_k\|_{L^p}^p\\
\leq  \bigl(\|\nabla(-\Delta)^{-1}(f_k-\div g_k)\|_{L^p}
+\|w_k-(-\Delta)^{-1}\nabla a_k\|_{L^p}  \bigr)\|w_k\|_{L^p}^{p-1}.
\end{multline}
In terms of $w,$ the function $a$ satisfies
the following \emph{damped} transport equation:
\begin{equation}\label{R-E55}
 \d_ta+\div(au)+a=-\div w.
 \end{equation}
Then, applying the operator $\d_i\dot{\Delta}_{k}$  to \eqref{R-E55}
and denoting $R_k^i\triangleq[u\cdot\nabla,\d_i\ddk]a$
gives
\begin{equation}\label{R-E56}
 \d_t\d_ia_k+u\cdot\nabla\d_ia_k+\d_ia_k=   -\d_i\ddk(a\div u)-\d_i\div w_k+R_k^i,\quad i=1,\cdots,d.\end{equation}
Multiplying  by $|\d_ia_k|^{p-2}\d_ia_k,$   integrating on $\R^d,$ and performing an integration by
parts in the second term of  \eqref{R-E56}, we get
$$\displaylines{
\frac1p\frac{d}{dt}\|\d_ia_k\|_{L^p}^p+\|\d_ia_k\|_{L^p}^p=\frac1p\int\div u\:|\d_ia_k|^p\,dx
\hfill\cr\hfill+\int\bigl(R_k^i -\d_i\ddk(a\div u)-\d_i\div w_k)|\d_ia_k|^{p-2}\d_ia_k\,dx.}
$$
Summing up on $i=1,\cdots,d,$ and applying H\"older  and Bernstein inequalities
 leads to
\begin{multline}\label{R-E57}
\frac1p\frac{d}{dt}\|\nabla a_k\|_{L^p}^p+\|\nabla a_k\|_{L^p}^p\leq\Bigl(\frac{1}{p}\|\mathrm{div}u\|_{L^\infty}\|\nabla a_k\|_{L^p}+\|\nabla\ddk(a\mathrm{div}u)\|_{L^p}\\+C2^{2k}\|w_k\|_{L^p}+\|R_k\|_{L^p}\Bigr)\|\nabla a_k\|_{L^p}^{p-1}.
\end{multline}
Adding up that inequality (multiplied by $\ep c_p$) to \eqref{R-E52} and \eqref{R-E54} yields
$$\displaylines{
\frac1p\frac{d}{dt}\bigl(\|\cP u_k\|_{L^p}^p+\|w_k\|_{L^p}^p+\ep c_p\|\nabla a_k\|_{L^p}^p\bigr)
+c_p2^{2k}\bigl(\mu_\infty\|\cP u_k\|_{L^p}^p+\|w_k\|_{L^p}^p)+\ep c_p\|\nabla a_k\|_{L^p}^p\hfill\cr\hfill
\leq  \bigl(\|\cP g_k\|_{L^p}+\|\nabla(-\Delta)^{-1}(f_k-\div g_k)\|_{L^p}\bigr)\|(\cP u_k,w_k)\|_{L^p}^{p-1}\hfill\cr\hfill
+\ep c_p\Bigl(\frac{1}{p}\|\mathrm{div}u\|_{L^\infty}\|\nabla a_k\|_{L^p}+\|\nabla\ddk(a\div u)\|_{L^p}+\|R_k\|_{L^p}\Bigr)\|\nabla a_k\|_{L^p}^{p-1}\hfill\cr\hfill
+C\ep c_p2^{2k}\|w_k\|_{L^p}\|\nabla a_k\|_{L^p}^{p-1}+\|w_k\|_{L^p}^{p}
+\|(-\Delta)^{-1}\nabla a_k\|_{L^p}\|w_k\|_{L^p}^{p-1}.}
$$
Taking advantage of Young inequality, we see  that the last line may be absorbed by the l.h.s. if
$\ep$ is taken small enough. It is also the case of the last two terms according to  \eqref{eq:B2}
(as  $(-\Delta)^{-1}$ is a homogeneous Fourier multiplier of degree $-2$),  if $k$ is large enough.
Therefore,  remembering that $f_k=\ddk\div(au)$ and using that
 $\nabla(-\Delta)^{-1}\div$ is a homogeneous  multiplier of degree $0,$
we  conclude that there exist some $k_0\in\Z$ and $c_0,\ep>0$ so that for all $k\geq k_0,$ we have
$$\displaylines{
\frac1p\frac{d}{dt}\bigl(\|\cP u_k\|_{L^p}^p+\|w_k\|_{L^p}^p+\ep c_p\|\nabla a_k\|_{L^p}^p\bigr)
+c_0\bigl(\|\cP u_k\|_{L^p}^p+\|w_k\|_{L^p}^p+\ep c_p\|\nabla a_k\|_{L^p}^p\bigr)
\hfill\cr\hfill
\leq C\bigl(\|g_k\|_{L^p}+\|\ddk(au)\|_{L^p}\bigr)\|(\cP u_k,w_k)\|_{L^p}^{p-1}\hfill\cr\hfill
+\ep c_p\Bigl(\frac{1}{p}\|\mathrm{div}u\|_{L^\infty}\|\nabla a_k\|_{L^p}+\|\nabla\ddk(a\div u)\|_{L^p}+\|R_k\|_{L^p}\Bigr)\|\nabla a_k\|_{L^p}^{p-1}.}
$$
Integrating in time, we arrive (taking smaller  $c_0$ as the case may be)  at
$$
e^{c_0t}\|(\cP u_k,w_k,\nabla a_k)(t)\|_{L^p}
\lesssim \|(\cP u_k,w_k,\nabla a_k)(0)\|_{L^p}+\int_0^te^{c_0\tau}S_k(\tau)\,d\tau
$$
with $S_k\triangleq S_k^1+\cdots+S_k^5$ and
$$\displaylines{
S_k^1\triangleq  \|\ddk(au)\|_{L^p},\quad
S_k^2\triangleq\|g_k\|_{L^p},\cr \quad S_k^3\triangleq\|\nabla\ddk(a\mathrm{div}\,u)\|_{L^p},\quad
S_k^4\triangleq \|R_k\|_{L^p}, \quad
S_k^5\triangleq\|\div u\|_{L^\infty}\|\nabla a_k\|_{L^p}.}$$
It is clear that $(u_k,\nabla a_k)$ satisfies a similar inequality, for  we have
\begin{eqnarray}\label{R-E62}
 u=w-\nabla(-\Delta)^{-1}a+\cP u
\end{eqnarray}
which leads  for $k\geq k_0$  to
\begin{eqnarray}\label{R-E63}
\|u_k-(w_k+\cP u_k)\|_{L^p} \lesssim 2^{-2k_0}\|\nabla a_k\|_{L^p}.
\end{eqnarray}Therefore, there exists a constant $c_{0}>0$ such that  for all $k\geq k_{0}$ and $t\geq0,$ we have
\begin{eqnarray}\label{R-E65}
\|(\nabla a_k,u_k)(t)\|_{L^p}\lesssim e^{-c_{0}t}\|(\nabla a_k(0),u_k(0))\|_{L^p}+\int_{0}^{t} e^{-c_{0}(t-\tau)}S_{k}(\tau)\,d\tau.
\end{eqnarray}
Now, multiplying both sides by $\langle t\rangle^{\alpha}2^{k(\frac dp-1)},$  taking the supremum on $[0,T],$
and summing up over $k\geq k_0$ yields
\begin{multline}\label{R-E66}
\|\langle t\rangle^\alpha(\nabla a,u)\|^h_{\wt L^\infty_T(\dot B^{\frac dp-1}_{p,1})}\lesssim
\|(\nabla a_0,u_0)\|_{\dot B^{\frac dp-1}_{p,1}}^h
\!\\+\sum_{k\geq k_0}\sup_{0\leq t\leq T}\biggl(\langle t\rangle^\alpha\!\int_0^t\!e^{c_0(\tau-t)}2^{k(\frac dp-1)}S_k\,d\tau\biggr)\cdotp
\end{multline}
In order to bound the sum, we first notice that
\begin{equation}\label{R-E67}
\sum_{k\geq k_0}\sup_{0\leq t\leq 2}\biggl(\langle t\rangle^\alpha\!\int_0^t\!e^{c_0(\tau-t)}2^{k(\frac dp-1)}S_k(\tau)\,d\tau\biggr)
\lesssim\int_0^2 \sum_{k\geq k_0}2^{k(\frac dp-1)}S_k(\tau)\,d\tau.
\end{equation}
It follows from Propositions \ref{prop2.4} and \ref{prop2.5} that
\begin{multline}\label{R-E68}
\int_0^2 \sum_{k\geq k_0}2^{k(\frac dp-1)}S_k(\tau)\,d\tau\lesssim\int_0^2 \!\Bigl(\|au\|^{h}_{\dot B^{\frac dp-1}_{p,1}}+\|g\|^{h}_{\dot B^{\frac dp-1}_{p,1}}+
\|\nabla u\|_{\dot B^{\frac dp}_{p,1}}\|a\|_{\dot B^{\frac dp}_{p,1}}\Bigr)d\tau.
\end{multline}
It is clear that the last term of the r.h.s.  may be  bounded by $C\cX^2_{p}(2)$ and that, owing to Prop. \ref{prop2.4},
we have \begin{equation}\label{R-E69}
\|au\|_{L^1_t(\dot B^{\frac dp-1}_{p,1})}^h\lesssim
\|au\|_{L^1_t(\dot B^{\frac dp}_{p,1})}\lesssim \|a\|_{L^2_t(\dot B^{\frac dp}_{p,1})}\|u\|_{L^2_t(\dot B^{\frac dp}_{p,1})}.
\end{equation}
Furthermore, combining Propositions \ref{prop2.4} and \ref{prop2.6} yields (remembering that $p<2d$)
$$\displaylines{
\|g\|_{L^1_t(\dot B^{\frac dp-1}_{p,1})}\lesssim \bigl(
 \|u\|_{L^\infty_t(\dot B^{\frac dp-1}_{p,1})}\|\nabla u\|_{L^1_t(\dot B^{\frac dp}_{p,1})}
+\|a\|_{L^\infty_t(\dot B^{\frac dp}_{p,1})}\|\nabla u\|_{L^1_t(\dot B^{\frac dp}_{p,1})}
 \hfill\cr\hfill+ \|a\|_{L^2_t(\dot B^{\frac dp}_{p,1})}\|\nabla a\|_{L^2_t(\dot B^{\frac dp-1}_{p,1})}\bigr).
}$$
Now, we observe that
\begin{eqnarray}\label{R-E71}
\|a\|_{L^2_t(\dot B^{\frac dp}_{p,1})}\leq \|a\|^{\ell}_{L^2_t(\dot B^{\frac dp}_{p,1})}+\|a\|^{h}_{L^2_t(\dot B^{\frac dp}_{p,1})}.\end{eqnarray}
Combining interpolation, H\"older inequality and embedding (here we use that $p\geq2$),
we may write
\begin{eqnarray}\label{R-E72}
\|a\|^{\ell}_{L^2_t(\dot B^{\frac dp}_{p,1})} &\lesssim& \Big(\|a\|^{\ell}_{L^{1}_t(\dot B^{\frac dp+1}_{p,1})}\Big)^{\frac{1}{2}}\Big(\|a\|^{\ell}_{L^{\infty}_t(\dot B^{\frac dp-1}_{p,1})}\Big)^{\frac{1}{2}}
\nonumber \\ &\lesssim& \|a\|^{\ell}_{L^{1}_t(\dot B^{\frac d2+1}_{2,1})}+\|a\|^{\ell}_{L^{\infty}_t(\dot B^{\frac d2-1}_{2,1})}\lesssim \cX_{p}(t).
\end{eqnarray}
Likewise, we have
$$
\|a\|^{h}_{L^2_t(\dot B^{\frac dp}_{p,1})}\lesssim \Big(\|a\|^{h}_{L^{1}_t(\dot B^{\frac dp}_{p,1})}\Big)^{\frac{1}{2}}\Big(\|a\|^{h}_{L^{\infty}_t(\dot B^{\frac dp}_{p,1})}\Big)^{\frac{1}{2}}\lesssim \cX_{p}(t).
$$
Arguing similarly for bounding $u,$ we get
\begin{equation}\label{R-E73}
\|a\|_{L^2_t(\dot B^{\frac dp}_{p,1})}+ \|u\|_{L^2_t(\dot B^{\frac dp}_{p,1})}\lesssim \cX_{p}(t),
\end{equation}
and one can conclude that the first two terms in the r.h.s. of \eqref{R-E68} may be bounded by $\cX_p^2(2).$
We thus have
\begin{equation}\label{R-E74}
\sum_{k\geq k_0}\sup_{0\leq t\leq 2}\langle t\rangle^\alpha\!\int_0^t\!e^{c_0(\tau-t)}2^{k(\frac dp-1)}S_k(\tau)\,d\tau
\lesssim \cX^2_{p}(2).
\end{equation}
Let us now bound the supremum for  $2\leq t\leq T$ in the last term of \eqref{R-E66}, assuming (with no loss
of generality) that $T\geq2.$
To this end, it is convenient to split the integral on $[0,t]$ into integrals
on  $[0,1]$ and $[1,t].$
The integral on $[0,1]$ is easy to handle: because $e^{c_0(\tau-t)}\leq e^{-c_{0}t/2}$ for $2\leq t \leq T$ and $0\leq \tau \leq1$,
 one can write that
 $$
 \begin{aligned}
\sum_{k\geq k_0}\sup_{2\leq t\leq T}\langle t\rangle^\alpha\!\int_0^1\!e^{c_0(\tau-t)}2^{k(\frac dp-1)}S_k(\tau)\,d\tau
&\leq\sum_{k\geq k_0}\sup_{2\leq t\leq T} \langle t\rangle^\alpha e^{-\frac{c_0}2t} \int_0^1 2^{k(\frac dp-1)}S_k\,d\tau\\
&\lesssim\int_0^1  \sum_{k\geq k_0}  2^{k(\frac dp-1)}S_k\,d\tau.
\end{aligned}
$$
Hence, following the   procedure leading to \eqref{R-E74}, we end up with
\begin{equation}\label{R-E76}
\sum_{k\geq k_0}\sup_{2\leq t\leq T}\biggl(\langle t\rangle^\alpha\!\int_0^1\!e^{c_0(\tau-t)}2^{k(\frac dp-1)}S_k(\tau)\,d\tau\biggr)\lesssim \cX^2_{p}(1).
\end{equation}
In order to bound  the $[1,t]$ part of the integral for $2\leq t\leq T,$ we notice that \eqref{R-E16} guarantees that
\begin{equation}\label{R-E76bis}
\sum_{k\geq k_0}\sup_{2\leq t\leq T}\biggl(\langle t\rangle^\alpha\!\int_1^te^{c_0(\tau-t)}2^{k(\frac dp-1)}S_k(\tau)\,d\tau\biggr)\lesssim
\sum_{k\geq k_0} 2^{k(\frac dp-1)}\sup_{1\leq t\leq T} t^\alpha S_k(t).
\end{equation}
In what follows, we shall   use repeatedly the following inequality
\begin{equation}\label{R-E77}
\|\tau\nabla u\|_{\wt L^\infty_t(\dot B^{\frac dp}_{p,1})}\lesssim \cD_{p}(t),
\end{equation}
which just stems  from the definition of $\cD_{p}(t),$ as regards the high-frequencies of $u$, and from
 Bernstein inequalities for the low frequencies. Indeed:
if $d\geq3$ then $d/2+1>2$ and one can write that
$$
\begin{aligned}
\|\tau\nabla u\|_{\widetilde{L}^\infty_t(\dot B^{\frac dp}_{p,1})}^\ell
&\lesssim \|\tau u\|_{\widetilde{L}^\infty_t(\dot B^{\frac d2+1}_{2,1})}^\ell
\lesssim \|\tau u\|_{{L}^\infty_t(\dot B^{2}_{2,1})}^\ell\\&\lesssim
 \|\langle \tau\rangle^{\frac{s_0}{2}+1} u\|_{L_t^\infty(\dot B^2_{2,1})}^\ell\leq \cD_p(t).
\end{aligned}
$$
In the 2D-case, observing that $p<4$ implies $s_0>0,$  we have for all small enough $\varepsilon>0$:
$$\begin{aligned}
\|\tau\nabla u\|_{\widetilde{L}^\infty_t(\dot B^{\frac 2p}_{p,1})}^\ell
&\lesssim \|\tau u\|_{\widetilde{L}^\infty_t(\dot B^{2}_{2,1})}^\ell
\lesssim \|\tau u\|_{{L}^\infty_t(\dot B^{2-2\varepsilon}_{2,1})}^\ell
\\&\lesssim \|\langle \tau\rangle^{\frac{s_0}{2}+1-\varepsilon} u\|_{L_t^\infty(\dot B^{2-2\varepsilon}_{2,1})}^\ell\leq \cD_p(t).
\end{aligned}
$$
To bound the contribution of  $S_k^1$ and $S_k^2$ in \eqref{R-E66}, we  use the fact that
\begin{equation}\label{R-E79}
\sum_{k\geq k_0} 2^{k(\frac dp-1)}\sup_{1\leq t\leq T} t^\alpha (S_k^1(t)+S_k^2(t))
\lesssim \|t^\alpha (au,g)\|_{\wt L^\infty_T(\dot B^{\frac dp-1}_{p,1})}^h.
\end{equation}
Now, product laws adapted to tilde spaces (see  Proposition \ref{prop2.4}) ensure that
\begin{eqnarray}\label{R-E80}
\|t^\alpha au^h\|_{\wt L^\infty_T(\dot B^{\frac dp-1}_{p,1})}&\!\!\!\lesssim\|a\|_{\wt L^\infty_T(\dot B^{\frac dp}_{p,1})}
\|t^{\alpha} u^h\|_{\wt L^\infty_T(\dot B^{\frac dp-1}_{p,1})}
\lesssim \cX_p(T)\cD_p(T),\\\label{R-E81}
\|t^\alpha a^hu^\ell\|_{\wt L^\infty_T(\dot B^{\frac dp-1}_{p,1})}&\!\!\!\lesssim\|t^\alpha a^h\|_{\wt L^\infty_T(\dot B^{\frac dp}_{p,1})}
\|u^\ell\|_{\wt L^\infty_T(\dot B^{\frac dp-1}_{p,1})}
\lesssim \cD_p(T)\cX_p(T).
\end{eqnarray}
Note that Bernstein inequality \eqref{eq:B1} and embedding imply that
$$\|t^\alpha a^\ell u^\ell\|_{\wt L^\infty_T(\dot B^{\frac dp-1}_{p,1})}^h\lesssim
\|t^{\alpha}a^\ell u^\ell\|_{\wt L^\infty_T(\dot B^{\frac d2}_{2,1})}.$$
Hence, using  Proposition \ref{prop2.4}, we discover that
\begin{equation}\label{R-E82}
\|t^\alpha a^\ell u^\ell\|_{\wt L^\infty_T(\dot B^{\frac dp-1}_{p,1})}\lesssim
\|t^{\alpha/2}a^\ell\|_{\wt L^\infty_T(\dot B^{\frac d2}_{2,1})}
\|t^{\alpha/2}u^\ell\|_{\wt L^\infty_T(\dot B^{\frac d2}_{2,1})}.
\end{equation}
Because  $\alpha\leq s_0+\min(2,\frac{d}2-\ep),$ we deduce that
\begin{eqnarray}\label{R-E83}
&&\|t^{\alpha/2}(a^\ell,u^\ell)\|_{\wt L^\infty_T(\dot B^{\frac d2}_{2,1})}
\lesssim
\|t^{\alpha/2}(a^\ell,u^\ell)\|_{L_T^\infty(\dot B^{\frac d2-\ep}_{2,1})}\leq \cD_p(T)\ \  \mbox{if}\ \  d\leq 4,\\\label{R-E84}
&&\|t^{\alpha/2}(a^\ell,u^\ell)\|_{\wt L^\infty_T(\dot B^{\frac d2}_{2,1})}\lesssim
\|t^{\alpha/2}(a^\ell,u^\ell)\|_{L_T^\infty(\dot B^{2}_{2,1})}\leq \cD_p(T)\ \  \mbox{if}\ \  d\geq 5.
 \end{eqnarray}
Therefore we conclude that
 \begin{equation}\label{R-E85}
\|t^\alpha(au)\|_{\wt L^\infty_T(\dot B^{\frac dp-1}_{p,1})}^h\lesssim \cD_{p}(T)(\cD_p(T)+\cX_p(T)).
\end{equation}
To bound the convection term of $g,$ we just write that
$$
\|t^\alpha (u\cdot\nabla u)\|_{\wt L^\infty_T(\dot B^{\frac dp-1}_{p,1})}^h\lesssim\|t^{\alpha-1} u\|_{\wt L^\infty_T(\dot B^{\frac {d}{p}-1}_{p,1})}
\|t \nabla u\|_{\wt L^\infty_T(\dot B^{\frac dp}_{p,1})}.
$$
On one hand, it  is obvious that $\|t^{\alpha-1} u\|^{h}_{\wt L^\infty_T(\dot B^{\frac {d}{p}-1}_{p,1})}\leq \cD_{p}(t).$
On the other hand, we have the following estimates  for $z=a,u$ and small enough $\ep$:
 \begin{eqnarray}\label{R-E86}
&&\|t^{\alpha-1}z\|^\ell_{\wt L_T^\infty(\dot B^{\frac dp-1}_{p,1})} \lesssim
\|t^{\alpha-1}z\|^\ell_{L_T^\infty(\dot B^{\frac d2-1-2\ep}_{2,1})}\leq \cD_p(T)\ \  \mbox{if}\ \  d\leq 6,\\\label{R-E87}
&&\|t^{\alpha-1}z\|^\ell_{\wt L_T^\infty(\dot B^{\frac dp-1}_{p,1})}  \lesssim
\|t^{\alpha-1}z\|^\ell_{L_T^\infty(\dot B^{2}_{2,1})}\leq \cD_p(T)  \ \  \mbox{if}\ \  d\geq 7,
 \end{eqnarray}
provided  $\alpha-1\leq\frac{s_0}{2}+\frac{d}{4}-\frac{1}{2}-\varepsilon$ if $d\leq 6$ and $\alpha-1\leq \frac{s_0}{2}+1$ if $d\geq 7.$
Hence \begin{equation}\label{R-E88}
\|t^\alpha (u\cdot\nabla u)\|_{\wt L^\infty_T(\dot B^{\frac dp-1}_{p,1})}^h\lesssim \cD^2_{p}(T).
\end{equation}
To bound the term with $k(a)\nabla a,$ we use that  according to Propositions \ref{prop2.4} and \ref{prop2.6},
and to \eqref{R-E83}, \eqref{R-E84}, we have
\begin{eqnarray}\label{R-E89}
&&\|t^\alpha(k(a)\nabla a^h)\|_{\wt L^\infty_T(\dot B^{\frac dp-1}_{p,1})}\lesssim \|a\|_{\wt L^\infty_T(\dot B^{\frac dp}_{p,1})}
\|t^\alpha a\|_{\wt L^\infty_T(\dot B^{\frac dp}_{p,1})}^h\leq \cX_{p}(T)\cD_{p}(T),\\\label{R-E90}
&&\|t^\alpha(k(a)\nabla a^\ell)\|_{\wt L^\infty_T(\dot B^{\frac dp-1}_{p,1})}\lesssim
\|t^{\alpha/2}a\|_{\wt L^\infty_T(\dot B^{\frac dp}_{p,1})}
\|t^{\alpha/2} a\|_{\wt L^\infty_T(\dot B^{\frac d2}_{2,1})}^\ell\lesssim \cD^2_{p}(T).
\end{eqnarray}
To bound the term containing $I(a)\cA u,$   we write that  \begin{equation}\label{R-E91}
\|t^\alpha I(a)\cA u\|_{\wt L^\infty_T(\dot B^{\frac dp-1}_{p,1})}
\lesssim\|t\nabla^2u\|_{\wt L^\infty_T(\dot B^{\frac dp-1}_{p,1})}\bigl(\|t^{\alpha-1}a\|_{\wt L^\infty_T(\dot B^{\frac d2}_{2,1})}^\ell
+\|t^{\alpha-1}a\|_{\wt L^\infty_T(\dot B^{\frac dp}_{p,1})}^h\bigr).
\end{equation}
The first term on the right-side may be bounded by virtue of \eqref{R-E77}, and it is  clear that
the last term is bounded by $\cD_{p}(T).$
As for  the second one, we use   \eqref{R-E86} and \eqref{R-E87}.
\medbreak
The last term of $g$ is of the type $\nabla F(a)\otimes \nabla u$ with $F(0)=0,$ and we have
$$
\|t^\alpha\nabla F(a)\otimes \nabla u\|_{\wt L^\infty_T(\dot B^{\frac dp-1}_{p,1})}
\lesssim \|t^{\alpha-1} a\|_{\wt L^\infty_T(\dot B^{\frac dp}_{p,1})} \|t\nabla u\|_{\wt L^\infty_T(\dot B^{\frac dp}_{p,1})}.
$$
So using \eqref{R-E77}, the definition of $\cD_p(T)$ and \eqref{R-E86}, \eqref{R-E87},  we see that
$$
\|t^\alpha\nabla F(a)\otimes \nabla u\|_{\wt L^\infty_T(\dot B^{\frac dp-1}_{p,1})}
\lesssim \cD_p^2(T).$$
Reverting to \eqref{R-E79}, we end up  with
\begin{equation}\label{R-E92}
\sum_{k\geq k_0}\sup_{1\leq t\leq T}t^\alpha2^{k(\frac dp-1)}(S_k^1+S_k^2)(t)
\lesssim  \cD_{p}(T)\cX_{p}(T)+\cD^2_{p}(T).\end{equation}
The term $S_k^3$ is similar to the last two terms of $g.$
As for  bounding  $S_k^4,$ we notice that a small modification of Proposition \ref{prop2.5}
(just include $t^\alpha$ in the definition of the commutator, follow the proof
treating the time variable as a parameter, and take the supremum on $[0,T]$ at the end) yields:
\begin{equation}\label{R-E92bis}
\sum_{k\in\Z}2^{k(\frac dp-1)}\sup_{0\leq t\leq T} t^\alpha\|R_k(t)\|_{L^p}\lesssim
 \|t\nabla u\|_{\wt L^\infty_T(\dot B^{\frac dp}_{p,1})}\|t^{\alpha-1}
 \nabla a\|_{\wt L^\infty_T(\dot B^{\frac dp-1}_{p,1})}.
 \end{equation}
 Hence using \eqref{R-E77}, \eqref{R-E86} and  \eqref{R-E87}   gives
 $$
\sum_{k\in\Z}2^{k(\frac dp-1)}\sup_{0\leq t\leq T} t^\alpha\|R_k(t)\|_{L^p}\lesssim \cD^2_{p}(T).
$$
The term with $S_k^5$ is clearly bounded by the r.h.s. of \eqref{R-E92bis}.
Putting all the above inequalities together, we conclude that
\begin{eqnarray}\label{R-E94}
\sum_{k\geq k_0}2^{k(\frac dp-1)}\sup_{1\leq t\leq T}t^\alpha S_k(t)\lesssim \cD_{p}(T)\cX_{p}(T)+\cD^2_{p}(T).
\end{eqnarray}
Plugging \eqref{R-E94} in  \eqref{R-E76bis}, and remembering   \eqref{R-E66}, \eqref{R-E74}  and \eqref{R-E76}, we end up with \begin{equation}\label{R-E95}
\|\langle t\rangle^\alpha(\nabla a,u)\|^h_{\wt L^\infty_T(\dot B^{\frac dp-1}_{p,1})}\lesssim
\|(\nabla a_0,u_0)\|^h_{\dot B^{\frac dp-1}_{p,1}}+\cX^2_{p}(T)+\cD^2_{p}(T).
\end{equation}


\subsubsection*{Step 3: Decay estimates with gain of regularity  for the high frequencies of $u$}
In order to bound the last term in $\cD_{p}(t),$ it suffices to notice  that the velocity $u$ satisfies
$$
\d_tu-\cA u=F\triangleq -(1+k(a))\nabla a-u\cdot\nabla u-I(a)\cA u+\frac1{1+a}\:\div\bigl(2\wt\mu(a)D(u)+\wt\lambda(a)\div u\,{\rm Id}\bigr).
$$
Hence
\begin{equation}\label{R-E97}
\d_t(t\cA u)-\cA(t\cA u)=\cA u+t\cA F.
\end{equation}
We thus deduce from Proposition \ref{prop2.7} and the remark that follows,  that
\begin{equation}\label{R-E98}
\|\tau\nabla^2u\|_{\wt L_t^\infty(\dot B^{\frac dp-1}_{p,1})}^h \lesssim \|\cA u\|_{L_t^1(\dot B^{\frac dp-1}_{p,1})}^h +\|\tau\cA F\|_{\wt L^\infty_t(\dot B^{\frac dp-3}_{p,1})}^h,
\end{equation}
whence, using the bounds given by Theorem \ref{thm1.1},
\begin{equation}\label{R-E99}
\|\tau\nabla u\|_{\wt L_t^\infty(\dot B^{\frac dp}_{p,1})}^h \lesssim
\|u\|^{h}_{L_t^1(\dot B^{\frac dp+1}_{p,1})}+\|\tau F\|_{\wt L^\infty_t(\dot B^{\frac dp-1}_{p,1})}^h
\lesssim  \cX_{p}(0)+\|\tau F\|_{\wt L^\infty_t(\dot B^{\frac dp-1}_{p,1})}^h.
\end{equation}
In order to bound the first term of $F,$  we notice that, because $\alpha\geq1,$ we have
\begin{eqnarray}\label{R-E100}
\|\tau\nabla a\|_{\wt L^\infty_t(\dot B^{\frac dp-1}_{p,1})}^h\lesssim
\|\langle\tau\rangle^\alpha a\|_{\wt L^\infty_t(\dot B^{\frac dp}_{p,1})}^h.
\end{eqnarray}
Next, product and composition estimates (see Propositions \ref{prop2.4} and \ref{prop2.6}) adapted to tilde spaces give
\begin{eqnarray}\label{R-E101}
\|\tau\,k(a)\nabla a\|_{\wt L^\infty_t(\dot B^{\frac dp-1}_{p,1})}^h\lesssim \|\tau^{\frac12} a\|_{\wt L^\infty_t(\dot B^{\frac dp}_{p,1})}^2\lesssim \cD^2_{p}(t),
\end{eqnarray}
as well as
\begin{equation}\label{R-E102}
\|\tau \, u\cdot\nabla u\|_{\wt L^\infty_t(\dot B^{\frac dp-1}_{p,1})}^h\lesssim \|u\|_{\wt L^\infty_t(\dot B^{\frac dp-1}_{p,1})}
\|\tau \nabla u\|_{\wt L^\infty_t(\dot B^{\frac dp}_{p,1})}\lesssim \cX_p(t)\cD_{p}(t)\end{equation}
and
\begin{equation}\label{R-E103}
\|\tau I(a)\cA u\|_{\wt L^\infty_t(\dot B^{\frac dp-1}_{p,1})}^h\lesssim \|a\|_{\wt L^\infty_t(\dot B^{\frac dp}_{p,1})}
\|\tau \nabla^2 u\|_{\wt L^\infty_t(\dot B^{\frac dp-1}_{p,1})}\lesssim \cX_p(t)\cD_{p}(t).
\end{equation}
Obviously, the terms $\frac{\wt\mu(a)}{1+a}\:\div D(u)$ and $\frac{\wt\lambda(a)}{1+a}\nabla\div u$
also satisfy \eqref{R-E103}.
Finally, we notice that for any smooth function $K,$ we have
$$
\|\tau K(a)\nabla a\otimes\nabla u\|_{\wt L^\infty_t(\dot B^{\frac dp-1}_{p,1})}^h\lesssim
(1+\|a\|_{\wt L^\infty_T(\dot B^{\frac dp}_{p,1})})\|\nabla a\|_{\wt L^\infty_T(\dot B^{\frac dp-1}_{p,1})}\|\tau\nabla u\|_{\wt L^\infty_T(\dot B^{\frac dp}_{p,1})}.
$$
Hence, reverting to \eqref{R-E99} and remembering \eqref{R-E77}, we get
\begin{eqnarray}\label{R-E104}
\|\tau\nabla u\|_{\wt L_t^\infty(\dot B^{\frac dp}_{p,1})}^h \lesssim  \cX_{p,0} +\cD_{p}(t)\cX_{p}(t)+\cD^2_{p}(t)
+\|\langle\tau\rangle^\alpha a\|_{\wt L^\infty_t(\dot B^{\frac dp}_{p,1})}^h.
\end{eqnarray}
Finally,  bounding the last term on the right-side of  \eqref{R-E104} according
to \eqref{R-E95},  and adding up the obtained  inequality to \eqref{s17} and \eqref{R-E95} yields for all $T\geq0,$
\begin{eqnarray}\label{R-E105}
\cD_{p}(T)\lesssim \cD_{p,0}+\|(a_0,u_0)\|^{\ell}_{\dot B^{\frac d2-1}_{2,1}}+ \|(\nabla a_0, u_0)\|_{\dot B^{\frac dp-1}_{p,1}}^h+ \cX^2_{p}(T)+\cD^2_{p}(T).
\end{eqnarray}
As Theorem \ref{thm1.1} ensures that $\cX_{p}\lesssim \cX_{p,0}\ll1$ and as
 $\|(a_0,u_0)\|^{\ell}_{\dot B^{\frac d2-1}_{2,1}}\lesssim\|(a_0,u_0)\|^{\ell}_{\dot B^{-s_0}_{2,\infty}},$
one can  conclude
that \eqref{R-E10} is fulfilled for all time if $\cD_{p,0}$ and $ \|(\nabla a_0, u_0)\|_{\dot B^{\frac dp-1}_{p,1}}^h$
are small enough. This completes the proof of Theorem \ref{thm2.1}.


\section{More decay estimates}\label{sec:4}

This short section is devoted to pointing out  some corollaries of Theorem \ref{thm2.1}.

To start with,  let us extend its statement to  general values of $\varrho_\infty,$ $c_\infty$ and $\nu_\infty.$
It is based on  the change of unknowns \eqref{eq:change}
and on  the scaling invariance \eqref{eq:scaling} of Besov norms.
For example, introducing the \emph{Mach number} $Ma\triangleq 1/c_\infty$ and the \emph{Reynolds number}
 $Re\triangleq \varrho_\infty/c_\infty,$ and denoting
$$\|z\|_{\dot B^s_{2,1}}^{\ell,\zeta}:= \sum_{2^k\leq\zeta2^{k_0}} 2^{ks}\|\ddk z\|_{L^2}\quad\hbox{for }\ \zeta>0,$$
 we easily find that
$$
\bigl\|\langle\wt\tau\rangle^{\frac{s_0+s}2}(\wt a,\wt u)\bigr\|^{\ell,1}_{L^\infty_{\wt t}(\dot B^s_{2,1})}
=\Bigl\|\Bigl\langle \frac{Re}{Ma^2}\,\tau\Bigr\rangle^{\frac{s_0+s}2}
\Bigl(\frac{\varrho-\varrho_\infty}{\varrho_\infty},Ma\, u\Bigr)\Bigr\|^{\ell,\frac{Re}{Ma}}_{L^\infty_{\frac{Re}{Ma}\wt t}
(\dot B^s_{2,1})},
$$
and similar relations for the other terms of $\cD_p(t).$
\medbreak
This leads to  the following statement:
\begin{thm}\label{thm4.1} Let $d,$ $p,$ $\alpha$  and $s_0$ be as in Theorem \ref{thm2.1}.
There exists a constant $c$ depending only on $p$ and $d$ such that if
$$
\Big\|\frac{\varrho_0-\varrho_\infty}{\varrho_\infty}\Big\|_{\dot{B}^{-s_{0}}_{2,\infty}}^{\ell,\frac{Re}{Ma}}
+Ma\|u_0\|_{\dot{B}^{-s_{0}}_{2,\infty}}^{\ell,\frac{Re}{Ma}}\leq c\biggl(\frac{Ma}{Re}\biggr)^{\frac{2d}p}
\quad\!\!\hbox{and}\!\!\quad
\frac1{\varrho_\infty}\|\nabla\varrho_0\|^{h,\frac{Re}{Ma}}_{\dot B^{\frac dp-1}_{p,1}}+Re\,\|u_0\|^{h,\frac{Re}{Ma}}_{\dot B^{\frac dp-1}_{p,1}}
\leq c_0,
$$
then  System \eqref{R-E1} has  a unique solution $(\varrho,u)$ satisfying the regularity
properties of Theorem \ref{thm1.1}. Furthermore, we have   for all $t\geq0,$
$$
\displaylines{
\sup_{s\in(-s_{0},2]}\Big\|\Big\langle\frac{Re}{Ma^2}\,\tau\rangle^{\frac {s_0+s}2}\Bigl(\frac{\varrho-\varrho_\infty}{\varrho_\infty},Ma\,u\Bigr)\Big\|_{L^\infty_t(\dot B^s_{2,1})}^{\ell,\frac{Re}{Ma}}
+\Big\|\Big\langle\frac{Re}{Ma^2}\tau\Big\rangle^{\alpha}\Bigl(\frac{\nabla\varrho}{\varrho_\infty},Re\,u\Bigr)\|_{\wt L^\infty_t(\dot B^{\frac dp-1}_{p,1})}^{h,\frac{Re}{Ma}}\hfill\cr\hfill
+\|\tau\nabla  u\|_{\wt L^\infty_t(\dot B^{\frac dp}_{p,1})}^{h,\frac{Re}{Ma}}
\lesssim \biggl(\frac{Re}{Ma}\biggr)^{\frac{2d}p}  \Big\|\Bigl(\frac{\varrho_0-\varrho_\infty}{\varrho_\infty},Ma\, u_0\Bigr)\Big\|_{\dot{B}^{-s_{0}}_{2,\infty}}^{\ell,\frac{Re}{Ma}}
+ \Bigl\|\Bigl(\frac{\nabla\varrho_0}{\varrho_\infty}, Re\, u_0\Bigr)\Bigr\|^{h,\frac{Re}{Ma}}_{\dot B^{\frac dp-1}_{p,1}}.
}
$$\end{thm}

Just to compare our results with those of the prior literature on decay estimates,
let us now state the $L^q-L^r$ type decay rates that we can get from our main theorem.
For notational simplicity, we assume that $\varrho_\infty=1$ and that $Re=Ma=1.$
\begin{cor}\label{cor1.1}  The solution $(\varrho,u)$  constructed in  Theorem \ref{thm2.1} satisfies
$$\displaylines{
\|\Lambda^{s}(\varrho-1)\|_{L^p}
\leq C\bigl(\cD_{p,0}
+\|(\nabla a_0,u_0)\|_{\dot B^{\frac dp-1}_{p,1}}^h\bigr)\langle t\rangle^{-\frac {s_0+s}2}
\  \hbox{ if } \  -s_0<s\leq\min\Bigl(2,\frac dp\Bigr),\cr
\|\Lambda^{s}u\|_{L^p}
\leq C\bigl(\cD_{p,0}+\|(\nabla a_0,u_0)\|_{\dot B^{\frac dp-1}_{p,1}}^h\bigr)\langle t\rangle^{-\frac {s_{0}+s}2}
\ \hbox{ if } \  -s_0<s\leq\min\Bigl(2,\frac dp-1\Bigr),}
$$
 where the fractional derivative
 operator $\Lambda^{\ell}$ is defined by $\Lambda^{\ell}f\triangleq\mathcal{F}^{-1}(|\cdot|^{\ell}\mathcal{F}f)$.
\end{cor}
\begin{proof}
Recall that for  functions with compactly supported Fourier transform,
 one has the embedding $\dot{B}^{s}_{2,1}\hookrightarrow\dot{B}^{s-d(1/2-1/p)}_{p,1} \hookrightarrow \dot{B}^{s}_{p,1}$
for $p\geq2.$
Hence, we may write
$$
\sup_{t\in[0,T]} \langle t\rangle^{\frac {s_{0}+s}2}\|\Lambda^s a\|_{\dot B^0_{p,1}}\lesssim
 \|\langle t\rangle^{\frac {s_{0}+s}2}a\|_{L^\infty_T(\dot B^s_{2,1})}^\ell
 +  \|\langle t\rangle^{\frac {s_{0}+s}2}a\|_{L^\infty_T(\dot B^s_{p,1})}^h.
$$
If follows from Inequality \eqref{R-E10} and the definition of $\cD_p$ and $\alpha$ that
$$
  \|\langle t\rangle^{\frac {s_{0}+s}2}a\|_{L^\infty_T(\dot B^s_{2,1})}^\ell\lesssim \cD_{p,0}
+\|(\nabla a_0,u_0)\|_{\dot B^{\frac dp-1}_{p,1}}^h
\quad\hbox{if }\ -s_0<s\leq 2
$$
and that, because we have $\alpha\geq\frac{s_0+s}2$ for all $s\leq\min(2,d/p),$
$$
  \|\langle t\rangle^{\frac {s_{0}+s}2}a\|_{L^\infty_T(\dot B^s_{p,1})}^h
  \lesssim \cD_{p,0}
+\|(\nabla a_0,u_0)\|_{\dot B^{\frac dp-1}_{p,1}}^h\quad\hbox{if }\ \  s\leq d/p.$$
This yields the desired result for $a.$
Bounding the velocity $u$ works almost the same, except that we need the stronger condition $s\leq d/p-1$
for the high frequencies.
This completes the proof of Corollary \ref{cor1.1}.
\end{proof}

\begin{rem}
Taking  $p=2$ (hence $s_0=d/2$) and $s=0$ in Corollary \ref{cor1.1} leads back to  the standard optimal   $L^{1}$-$L^{2}$ decay rate of $(a,u).$ Note however that our  estimates also hold in the  general $L^p$ critical framework. Additionally,
the regularity  index $s$ can take both negative and nonnegative values, rather than only nonnegative integers, which   improves the classical decay results  in high Sobolev regularity, such as \cite{MN2} or \cite{P}.
\end{rem}

One can get   more $L^{q}$-$L^{r}$ decay estimates, as a consequence
of the following  Gagliardo-Nirenberg type inequalities
which parallel the work of Sohinger and Strain \cite{SS} (see also  \cite{BCD}, Chap. 2, and  \cite{XK2}):
\begin{prop}\label{prop2.10}
The following interpolation inequality  holds true:
\begin{eqnarray*}
\|\Lambda^{\ell}f\|_{L^{r}}\lesssim \|\Lambda^{m}f\|^{1-\theta}_{L^{q}}\|\Lambda^{k}f\|^{\theta}_{L^{q}},
\end{eqnarray*}
whenever  $0\leq\theta\leq1$,  $1\leq q\leq r\leq\infty$ and $$\ell+d\Big(\frac{1}{q}-\frac{1}{r}\Big)=m(1-\theta)+k\theta.$$
\end{prop}

\begin{cor}\label{cor1.2}
Let the assumptions of Theorem \ref{thm2.1} be fulfilled with $p=2$. Then the corresponding solution $(\varrho,u)$  satisfies
\begin{eqnarray}
\|\Lambda^{\ell}(\varrho-1,u)\|_{L^r}\leq C\bigl(\cD_{2,0}
+\|(\nabla a_0,u_0)\|_{\dot B^{\frac d2-1}_{2,1}}^h\bigr)\bigl \langle t\rangle^{-\frac d2(1-\frac1r)-\frac\ell2}, \label{R-E13}
\end{eqnarray}
for all $2\leq r\leq\infty$ and $\ell\in\R$ satisfying
$-\frac d2<\ell+d\big(\frac12-\frac1r\big) < \min\big(2,\frac d2-1\big)\cdotp$
\end{cor}
\begin{proof}
It follows from Corollary \ref{cor1.1}  with $p=2,$ and Proposition \ref{prop2.10} with $q=2,$ $m=\min(2,\frac d2-1)$
and $k=-\frac d2+\ep$ with  $\ep$ small enough. 
Indeed, if we define  $\theta$ by the relation 
$$
k\theta+m(1-\theta)=\ell+d\Bigl(\frac12-\frac1r\Bigr),
$$
then one can take  $\ep$ so small as $\theta$ to be  in $(0,1).$ Therefore we have
\begin{eqnarray}\label{R-E106}
\|\Lambda^{\ell}(a,u)\|_{L^{r}} &\lesssim &\|\Lambda^{m}(a,u)\|_{L^2}^{1-\theta}\|\Lambda^{k}(a,u)\|^{\theta}_{L^2}
\nonumber\\& \lesssim & \bigl(\cD_{2,0}
+\|(\nabla a_0,u_0)\|_{\dot B^{\frac d2-1}_{2,1}}^h\bigr)\Big\{\langle t\rangle^{-\frac d4-\frac m2}\Big\}^{1-\theta}
\Big\{\langle t\rangle^{-\frac d4-\frac k2}\Big\}^{\theta}
\nonumber\\&=& \bigl(\cD_{2,0}
+\|(\nabla a_0,u_0)\|_{\dot B^{\frac d2-1}_{2,1}}^h\bigr)\langle t\rangle^{-\frac {d}{4}-\frac{m}{2}(1-\theta)-\frac{k}{2}\theta},
\end{eqnarray}
 which completes the proof of the corollary.
\end{proof}


\begin{appendix}
\section{Littlewood-Paley decomposition and Besov spaces}\setcounter{equation}{0}

We here recall   basic properties
 of Besov spaces and paradifferential calculus that have been used repeatedly in the paper
 (more details may be found  in e.g. Chap. 2 and 3 of \cite{BCD}).
  We also prove  some  slightly less classical product laws, and  the commutator estimate   \eqref{R-E92bis}.
\medbreak
As mentioned in the introduction, homogeneous  Besov spaces possess
 scaling invariance properties. In the $\R^d$ case, they read for any $\sigma\in\R$ and $(p,r)\in[1,+\infty]^2$:
\begin{equation}\label{eq:scaling}
C^{-1}\lambda^{\sigma-\frac dp} \|f\|_{\dot B^\sigma_{p,r}}\leq
\|f(\lambda\cdot)\|_{\dot B^\sigma_{p,r}}\leq C\lambda^{\sigma-\frac dp}  \|f\|_{\dot B^\sigma_{p,r}},\qquad\lambda>0,
\end{equation}
where the constant $C$ depends only on $\sigma,$ $p$ and on the dimension $d.$
\medbreak
The following embedding properties have been used several times:
\begin{itemize}
  \item For any $p\in[1,\infty]$ we have the  continuous embedding
$$\dot B^0_{p,1}\hookrightarrow L^p\hookrightarrow \dot B^0_{p,\infty}.$$
\item If  $\sigma\in\R,$ $1\leq p_1\leq p_2\leq\infty$ and $1\leq r_1\leq r_2\leq\infty,$
  then $\dot B^{\sigma}_{p_1,r_1}\hookrightarrow
  \dot B^{\sigma-d(\frac1{p_1}-\frac1{p_2})}_{p_2,r_2}.$
  \item The space  $\dot B^{\frac dp}_{p,1}$ is continuously embedded in   the set  of
bounded  continuous functions (going to $0$ at infinity if    $p<\infty$).
  \end{itemize}
Let us also mention the following  interpolation inequality
 that is   satisfied whenever
   $1\leq p,r_1,r_2,r\leq\infty,$ $\sigma_1\not=\sigma_2$ and $\theta\in(0,1)$:
  $$
  \|f\|_{\dot B^{\theta\sigma_2+(1-\theta)\sigma_1}_{p,r}}\lesssim\|f\|_{\dot B^{\sigma_1}_{p,r_1}}^{1-\theta}
  \|f\|_{\dot B^{\sigma_2}_{p,r_2}}^\theta.
  $$
The following  product estimates in Besov spaces    play a fundamental  role in our analysis of  the bilinear
terms of \eqref{R-E3}.
\begin{prop}\label{prop2.4}
Let $\sigma>0$ and $1\leq p,r\leq\infty$. Then $\dot{B}^{\sigma}_{p,r}\cap L^{\infty}$ is an algebra and
$$
\|fg\|_{\dot{B}^{\sigma}_{p,r}}\lesssim \|f\|_{L^{\infty}}\|g\|_{\dot{B}^{\sigma}_{p,r}}+\|g\|_{L^{\infty}}\|f\|_{\dot{B}^{\sigma}_{p,r}}.
$$
Let the real numbers $\sigma_{1},$ $\sigma_{2},$ $p_1$  and $p_2$ be such that
$$
\sigma_1+\sigma_2>0,\quad \sigma_1\leq\frac d{p_1},\quad\sigma_2\leq\frac d{p_2},\quad
\sigma_1\geq\sigma_2,\quad\frac1{p_1}+\frac1{p_2}\leq1.
$$
Then we have
$$\|fg\|_{\dot{B}^{\sigma_{2}}_{q,1}}\lesssim \|f\|_{\dot{B}^{\sigma_{1}}_{p_1,1}}\|g\|_{\dot{B}^{\sigma_{2}}_{p_2,1}}\quad\hbox{with}\quad
\frac1{q}=\frac1{p_1}+\frac1{p_2}-\frac{\sigma_1}d\cdotp$$
Finally, for exponents $\sigma>0$ and $1\leq p_1,p_2,q\leq\infty$ satisfying
$$
\frac{d}{p_1}+\frac{d}{p_2}-d\leq \sigma \leq\min\biggl(\frac d{p_1},\frac d{p_2}\biggr)\quad\hbox{and}\quad \frac 1q=\frac 1{p_1}+\frac 1{p_2}-\frac\sigma d,
$$
we have
$$\|fg\|_{\dot{B}^{-\sigma}_{q,\infty}}\lesssim
\|f\|_{\dot{B}^{\sigma}_{p_1,1}}\|g\|_{\dot{B}^{-\sigma}_{p_2,\infty}}.
$$
\end{prop}
\begin{proof}
The first inequality is classical (see  e.g. \cite{BCD}, Chap. 2).
For proving the second item, we need the following so-called Bony decomposition  for the product of two tempered
distributions $f$ and $g$:
\begin{equation}\label{eq:bony}
fg=T_fg+R(f,g)+T_{g}f,
\end{equation}
where  the \emph{paraproduct} between  $f$ and $g$ is defined by
$$T_fg:=\sum_j \dot S_{j-1} f\ddj g\quad\hbox{with }\ \dot S_{j-1}\triangleq\chi(2^{-(j-1)}D),$$
and  the \emph{remainder} $R(f,g)$ is given by the series:
$$
R(f,g):=\sum_j \ddj f\, (\dot\Delta_{j-1}g+\ddj g+\dot\Delta_{j+1}g\bigr).
$$
In the case $\sigma_2\geq0$ then we use  the embeddings $\dot B^{\sigma_1}_{p_1,1}\hookrightarrow L^{q_1}$ with
$\frac1{q_1}=\frac1{p_1}-\frac{\sigma_1}d$
and  $\dot B^{\sigma_2}_{p_2,1}\hookrightarrow L^{q_2}$ with
$\frac1{q_2}=\frac1{p_2}-\frac{\sigma_2}d$ and the fact that:
\begin{itemize}
\item $T$ maps $L^{q_1}\times \dot B^{\sigma_2}_{p_2,1}$ to $\dot B^{\sigma_2}_{q,1}$ with $\frac 1q=\frac1{q_1}+\frac{1}{p_2}
=\frac1{p_1}+\frac1{p_2}-\frac{\sigma_1}d$;
\item $T$ maps $L^{q_2}\times \dot B^{\sigma_1}_{p_1,1}$ to $\dot B^{\sigma_1}_{q_3,1}$ with $\frac 1{q_3}=\frac1{q_2}+\frac{1}{p_1}
=\frac1{p_1}+\frac1{p_2}-\frac{\sigma_2}d\cdotp$
\end{itemize}
As $\sigma_2\leq\sigma_1,$ we have $q_3\leq q,$ and thus
  $\dot B^{\sigma_1}_{q_3,1}\hookrightarrow \dot B^{\sigma_2}_{q,1}.$
Therefore the two paraproduct terms of \eqref{eq:bony} fulfill the desired inequality.

To bound the remainder term $R(f,g),$ we use the continuity result
$\dot B^{\sigma_1}_{p_1,1}\times\dot B^{\sigma_2}_{p_2,1}\to \dot B^{\sigma_1+\sigma_2}_{p,1}$
with $\frac1p=\frac1{p_1}+\frac1{p_2},$ and the embedding
$\dot B^{\sigma_1+\sigma_2}_{p,1}\hookrightarrow\dot B^{\sigma_2}_{q,1}.$

In the case $\sigma_2<0$ we use the fact that $T$ maps
$\dot B^{\sigma_2}_{p_2,1}\times\dot B^{\sigma_1}_{p_1,1}$ to $\dot B^{\sigma_1+\sigma_2}_{p,1}$
and, again, the embedding $\dot B^{\sigma_1+\sigma_2}_{p,1}\hookrightarrow\dot B^{\sigma_2}_{q,1}.$
\medbreak
Let us finally prove the last item.
To this end, we use the fact that both
$R$ and $T$ map $\dot B^{-\sigma}_{p_2,\infty}\times \dot B^{\sigma}_{p_1,1}$ to $\dot B^0_{p,\infty},$ with
$\frac1p=\frac1{p_1}+\frac 1{p_2}\cdotp$ As $\dot B^0_{p,\infty}$ continuously embeds in $\dot{B}^{-\sigma}_{q,\infty},$
the last two terms of \eqref{eq:bony} satisfy the desired inequality.
Regarding the first term in \eqref{eq:bony}, it suffices to notice that, as $0<\sigma\leq d/p_1,$ we have
 $\dot B^{\sigma}_{p_1,1}\hookrightarrow L^{q_1}$ with $\frac d{q_1}=\frac d{p_1}-\sigma,$
 and that $T$ maps $L^{q_1}\times \dot B^{-\sigma}_{p_2,\infty}$ to $\dot B^{-\sigma}_{q,\infty}.$
\end{proof}

To handle the case $p>d$ in the proof of Theorem \ref{thm2.1}, just resorting to the above proposition
does not allow to get suitable bounds for the low frequency part of some nonlinear terms.
We had to take advantage of the following result.
\begin{prop}\label{prop2.4b} Let $k_0\in\Z,$ and denote $z^\ell\triangleq\dot S_{k_0}z,$  $z^h\triangleq z-z^\ell$ and, for any $s\in\R,$
$$
\|z\|_{\dot B^s_{2,\infty}}^\ell\triangleq\sup_{k\leq k_0}2^{ks} \|\ddk z\|_{L^2}.
$$
There exists a universal integer $N_0$ such that  for any $2\leq p\leq 4$ and $\sigma>0,$ we have 
\begin{eqnarray}\label{eq:est1}
&&\|f g^h\|_{\dot B^{-s_0}_{2,\infty}}^\ell\leq C \bigl(\|f\|_{\dot B^\sigma_{p,1}}+\|\dot S_{k_0+N_0}f\|_{L^{p^*}}\bigr)\|g^h\|_{\dot B^{-\sigma}_{p,\infty}}\\\label{eq:est2}
&&\|f^h g\|_{\dot B^{-s_0}_{2,\infty}}^\ell
\leq C \bigl(\|f^h\|_{\dot B^\sigma_{p,1}}+\|(\dot S_{k_0+N_0}-\dot S_{k_0})f\|_{L^{p^*}}\bigr)\|g\|_{\dot B^{-\sigma}_{p,\infty}}
\end{eqnarray}
with  $s_0\triangleq \frac{2d}p-\frac d2$ and $\frac1{p^*}\triangleq\frac12-\frac1p,$
and $C$ depending only on $k_0,$ $d$ and $\sigma.$
\end{prop}
\begin{proof}
To prove the first inequality, we start with Bony's decomposition:
$$
fg^h= T_{g^h}f+R(g^h,f)+T_{f}g^h.
$$
As $-\sigma<0,$ the first two terms are in $\dot B^0_{p/2,\infty}$ and thus
in $\dot B^{-s_0}_{2,\infty}$ by embedding. Moreover,
$$
\|T_{g^h}f+R(g^h,f)\|_{\dot B^{-s_0}_{2,\infty}}\lesssim \|f\|_{\dot B^\sigma_{p,1}}\|g^h\|_{\dot B^{-\sigma}_{p,\infty}}.
$$
To handle the last term, we notice that the definition of the spectral  truncation operators $\ddk$ and $\dot S_k$ implies that
$\ddk g^h\equiv0$ if $k<k_0-1.$ As in addition $\dot\Delta_{k'}(\dot S_{k-1}f\,\ddk g^h)\equiv0$ if $|k-k'|>4,$
we deduce that
$$\begin{aligned}
\|T_fg^h\|^\ell_{\dot B^{-s_0}_{2,\infty}} &= \sup_{k\leq k_0}2^{-ks_0} \biggl\|\ddk\biggl(\sum_{k'\geq k_0-1} \dot S_{k'-1}f\,\dot\Delta_{k'}g^h\biggr)\biggr\|_{L^2}\\
&\leq C 2^{-k_0s_0} \sup_{k_0-1\leq k'\leq k_0+4} \|\dot S_{k'-1} f\|_{L^{p^*}} \|\dot\Delta_{k'}g^h\|_{L^p}.\end{aligned}
$$
This gives \eqref{eq:est1}.
\medbreak
Proving the second inequality is similar. On one hand, as above, we have
$$
\|T_{g}f^h+R(g,f^h)\|_{\dot B^{-s_0}_{2,\infty}}\lesssim \|f^h\|_{\dot B^\sigma_{p,1}}\|g\|_{\dot B^{-\sigma}_{p,\infty}}.
$$
On the other hand, owing to the definition of $f^h$ and of $\dot S_{k'-1},$ 
$$
T_{f^h}g=\sum_{k'\geq k_0+1} \dot S_{k'-1}f^h\,\dot\Delta_{k'} g
$$
and thus  $\ddk T_{f^h}g\equiv0$ for $k<k_0-4.$
\medbreak
Therefore we have
$$\|T_{f^h}g\|_{\dot B^{-s_0}_{2,\infty}}^\ell\lesssim \|(\dot S_{k_0+4}-\dot S_{k_0})f\|_{L^{p^*}}\sum_{|k-k_0|\leq4}
\|\ddk g\|_{L^p},
$$
whence \eqref{eq:est2}.
\end{proof}

System \eqref{R-E3} also involves compositions of functions
(through $I(a),$ $\wt \lambda(a),$ $\wt\mu(a)$ and $G'(a)$) that
are  bounded thanks to the following classical result:
\begin{prop}\label{prop2.6}
Let $F:\R\rightarrow\R$ be  smooth
with $F(0)=0.$
For  all  $1\leq p,r\leq\infty$ and  $\sigma>0$ we have
$F(f)\in \dot B^\sigma_{p,r}\cap L^\infty$  for  $f\in \dot B^\sigma_{p,r}\cap L^\infty,$  and
\begin{eqnarray*}
\|F(f)\|_{\dot B^\sigma_{p,r}}\leq C\|f\|_{\dot B^\sigma_{p,r}}
\end{eqnarray*}
with $C$ depending only on $\|f\|_{L^\infty},$ $F'$ (and higher derivatives),  $\sigma,$ $p$ and $d.$
\end{prop}
Let us now  recall the following classical \emph{Bernstein inequality}:
\begin{equation}\label{eq:B1}
\|D^kf\|_{L^{b}}
\leq C^{1+k} \lambda^{k+d(\frac{1}{a}-\frac{1}{b})}\|f\|_{L^{a}}
\end{equation}
that holds  for all function $f$ such that
$\mathrm{Supp}\,\mathcal{F}f\subset \{\xi\in \mathbb{R}^{d}: |\xi|\leq R\lambda\}$ for some $R>0$
and $\lambda>0,$  if  $k\in\N$ and  $1\leq a\leq b\leq\infty$.
\medbreak
More generally, if assume  $f$ to satisfy  $\mathrm{Supp}\,\mathcal{F}f\subset \{\xi\in \mathbb{R}^{d}:
R_{1}\lambda\leq|\xi|\leq R_{2}\lambda\}$ for some $0<R_1<R_2$  and $\lambda>0,$
then for any smooth  homogeneous of degree $m$
function $A$ on $\R^d\setminus\{0\}$ and   $1\leq a\leq\infty,$ we have
(see e.g. Lemma 2.2 in \cite{BCD}):
 \begin{equation}\label{eq:B2}
\|A(D)f\|_{L^{a}}\lesssim\lambda^{m}\|f\|_{L^{a}}.
\end{equation}
An obvious  consequence of \eqref{eq:B1} and \eqref{eq:B2} is that
$\|D^kf\|_{\dot{B}^s_{p, r}}\thickapprox \|f\|_{\dot{B}^{s+k}_{p, r}}$ for all $k\in\N.$
\medbreak
We  also need the following nonlinear generalization of \eqref{eq:B2} (see Lemma 8 in \cite{D1}):
\begin{prop}\label{prop2.3bis}
If $\mathrm{Supp}\,\mathcal{F}f\subset \{\xi\in \mathbb{R}^{d}:
R_{1}\lambda\leq|\xi|\leq R_{2}\lambda\}$ then there exists $c$ depending only on $d,$ $R_1$  and $R_2$
so that for all $1<p<\infty,$
$$c\lambda^2\biggl(\frac{p-1}p\biggr)\int_{\R^d}|f|^p\,dx\leq (p-1)\int_{\R^d}|\nabla f|^2|f|^{p-2}\,dx
=-\int_{\R^d}\Delta f\: |f|^{p-2}f\,dx.
$$
\end{prop}
A time-dependent version of  the following  commutator estimate has been used  in the second step of the proof of Theorem \ref{thm2.1}.
\begin{prop}\label{prop2.5}
Let  $1\leq p, p_{1}\leq\infty$ and
\begin{equation}\label{eq:cond}-\min\biggl(\frac{d}{p_{1}},\frac{d}{p'}\biggr)<\sigma\leq1+\min\biggl(\frac dp,\frac{d}{p_{1}}\biggr)\cdotp\end{equation}
There exists a
constant $C>0$ depending only on $\sigma$ such that for all $j\in\Z$ and $\ell\in\{1,\cdots,d\},$ we have
\begin{equation}\label{eq:com}\|[v\cdot\nabla,\d_\ell\dot{\Delta}_{j}]a\|_{L^{p}}\leq
Cc_{j}2^{-j(\sigma-1)}\|\nabla v\|_{\dot{B}^{\frac{d}{p_{1}}}_{p_{1},1}}\|\nabla a\|_{\dot{B}^{\sigma-1}_{p,1}},
\end{equation}
 where the commutator
$[\cdot,\cdot]$ is defined by $[f,g]=fg-gf$ and $(c_{j})_{j\in\Z}$ denotes
a sequence such that $\|(c_{j})\|_{ {\ell^{1}}}\leq 1$.
\end{prop}
\begin{proof} We just  sketch the proof as it  is very similar to that of Lemma 2.100 in \cite{BCD}.
Decomposing the two terms of $[v\cdot\nabla,\d_\ell\dot{\Delta}_{j}]a$ according to \eqref{eq:bony},
we see  that, with the summation convention over repeated indices:
\begin{multline}\label{eq:Rj}
[v\cdot\nabla,\d_\ell\dot{\Delta}_{j}]a= [T_{v^{k}},\d_\ell\dot\Delta_{j}]\partial_{k}a
+T_{\partial_{k}\d_\ell\dot\Delta_{j}a}v^{k}
\\-\d_\ell\dot\Delta_{j}T_{\partial_{k}a}v^{k}
+R(v^{k},\partial_k\d_\ell\dot\Delta_{j}a)
-\d_\ell\dot\Delta_{j}R(v^{k},\d_ka).
\end{multline}
To bound the  first term, we notice that owing to the properties of spectral localization of the Littlewood-Paley decomposition, we have
  $$ [T_{v^{k}},\d_\ell\dot\Delta_{j}]\partial_{k}a=2^{j}\sum_{|j-j'|\leq
4}[\dot S_{j'-1}v^{k},\wt\Delta_{j}^\ell]\partial_{k}\dot\Delta_{j'}a\quad\hbox{with}\quad
\wt\Delta_{j}^\ell\triangleq i(\xi^\ell\varphi)(2^{-j}D).
$$
Now, setting $h^\ell:=\cF^{-1}(i\xi^\ell\varphi),$ we get
$$
[\dot S_{j'-1}v^{k},\wt\Delta_{j}^\ell]\partial_{k}\dot\Delta_{j'}a(x) =
2^{jd}\!\int_{\R^d}\! h^\ell(2^j(x-y))
\bigl(\dot S_{j'-1}v^{k}(x)-\dot S_{j'-1}v^{k}(y)\bigr)\d_k\dot\Delta_{j'} a(y)\,dy.
$$
Hence using   the mean value formula and Bernstein inequalities yields
$$
\| [T_{v^{k}},\d_\ell\dot\Delta_{j}]\partial_{k}a\|_{L^p}\lesssim   \|{\nabla v}\|_{L^\infty}
\sum_{|j'-j|\leq4}\|{\dot\Delta_{j'}\nabla a}\|_{L^p}.$$
As $\dot B^{\frac d{p_1}}_{p_1,1}\hookrightarrow L^\infty,$ we get \eqref{eq:com} for that term.
\medbreak
Bounding the  third and last term in \eqref{eq:Rj}  follows from
standard results of continuity for the remainder and paraproduct operators.
Here we need  Condition \eqref{eq:cond}.
To estimate the second term of \eqref{eq:Rj},  let us write that
$$
T_{\partial_{k}\d_\ell\dot\Delta_{j}a}v^{k}=\sum_{j'\geq j-3}\dot S_{j'-1}\partial_k\d_\ell\ddj a\,\dot\Delta_{j'}v^k.
$$
Using Bernstein inequality, this yields
$$
\|T_{\partial_{k}\d_\ell\dot\Delta_{j}a}v^{k}\|_{L^p} \lesssim \sum_{j'\geq j-3}
2^{j-j'}\|\nabla\ddj a\|_{L^p}\,\|\nabla\dot\Delta_{j'}v\|_{L^\infty},
$$
and  convolution inequality for series and embedding thus ensures \eqref{eq:com}.
\medbreak
Finally, we have (because $\dot\Delta_{j'}\ddj=0$ if $|j'-j|>1$),
$$
R(v^{k},\partial_k\d_\ell\dot\Delta_{j}a)=\sum_{|j'-j|\leq1}  (\dot\Delta_{j'-1}\!+\!\dot\Delta_{j'}\!+\!\dot\Delta_{j'+1})v^k\:
\dot\Delta_{j'}\ddj\d_\ell\d_ka.
$$
Hence, by virtue of Bernstein inequality,
$$
\|R(v^{k},\partial_k\d_\ell\dot\Delta_{j}a)\|_{L^p} \lesssim  \sum_{|j'-j|\leq1} \|\dot\Delta_{j} \d_ka\|_{L^p} \|\nabla v^k\|_{L^\infty},
$$
which completes the proof of \eqref{eq:com}.
\end{proof}
\smallbreak
When localizing PDE's by means of  Littlewood-Paley decomposition,  one  naturally
ends up with  bounds for each dyadic block in spaces of type $L^{\rho}_{T}(L^{p})\triangleq L^\rho(0,T;L^p(\R^d)).$
To get an information on Besov norms, we then have to  perform a summation on $\ell^{r}(\mathbb{Z})$.
However, this does not quite yield a  bound for the norm in $L^{\rho}_{T}(\dot{B}^{\sigma}_{p,r})$, as the time integration has been performed \emph{before} the $\ell^{r}$ summation. This leads to the definition of
the following norms first introduced by J.-Y. Chemin in \cite{Chemin}  (see also \cite{ChL} for the particular case of Sobolev spaces)
for $0\leq T\leq+\infty,$ $\sigma\in\mathbb{R}$ and  $1\leq r,p,\rho\leq\infty$:
$$\|f\|_{\widetilde{L}^{\rho}_{T}(\dot{B}^{\sigma}_{p,r})}\triangleq\Big\|\bigl(2^{j\sigma}\|\dot{\Delta}_{j}f\|_{L^{\rho}_{T}(L^{p})}\bigr)\Big\|_{\ell^r(\Z)}.$$
For notational simplicity, index $T$ is omitted if $T=+\infty.$
\smallbreak
We also used the following functional space:
\begin{equation}\label{eq:tilde}\widetilde{\mathcal{C}}_{b}(\R_+;\dot{B}^{\sigma}_{p,r})\triangleq\bigl\{f\in \mathcal{C}(\R_+;\dot{B}^{\sigma}_{p,r})\ s.t.\
\|f\|_{\widetilde{L}^{\infty}(\dot{B}^{\sigma}_{p,r})}<\infty\bigr\}\cdotp
\end{equation}
The above norms  may be compared with those of the
more standard  Lebesgue-Besov spaces $L^{\rho}_{T}(\dot{B}^{\sigma}_{p,r})$ via
 Minkowski's inequality:
\begin{equation}\label{eq:mink}
\|f\|_{\widetilde{L}^{\rho}_{T}(\dot{B}^{\sigma}_{p,r})}\leq\|f\|_{L^{\rho}_{T}(\dot{B}^{\sigma}_{p,r})}\,\,\,
\mbox{if }\,\, r\geq\rho,\ \ \ \
\|f\|_{\widetilde{L}^{\rho}_{T}(\dot{B}^{\sigma}_{p,r})}\geq\|f\|_{L^{\rho}_{T}(\dot{B}^{\sigma}_{p,r})}\,\,\,
\mbox{if }\,\, r\leq\rho.
\end{equation}
The general principle is that all the properties of continuity for the product, commutators and composition which are true for Besov norms  extend to  the above norms:  the time exponent $\rho$ just behaves according to H\"{o}lder inequality.
\medbreak
Using the norms defined in \eqref{eq:tilde} leads to optimal regularity estimates for the heat equation,
as is recalled in the proposition below.
\begin{prop}\label{prop2.7}
Let $\sigma\in \mathbb{R}$, $(p,r)\in [1,\infty]^2$ and $1\leq \rho_{2}\leq\rho_{1}\leq\infty$.  Let $u$  satisfy
$$\left\{\begin{array}{lll}\d_tu-\mu\Delta u=f,\\
u_{|t=0}=u_0.\end{array}
\right.$$
Then  for all $T>0$ the following a priori estimate is fulfilled:
\begin{equation}\label{R-E14}\mu^{\frac1{\rho_1}}\|u\|_{\tilde L_{T}^{\rho_1}(\dot B^{\sigma+\frac2{\rho_1}}_{p,r})}\lesssim
\|u_0\|_{\dot B^\sigma_{p,r}}+\mu^{\frac{1}{\rho_2}-1}\|f\|_{\tilde L^{\rho_2}_{T}(\dot B^{\sigma-2+\frac2{\rho_2}}_{p,r})}.
\end{equation}
\end{prop}
\begin{rem} The solutions to the following \emph{Lam\'e system}
\begin{equation}\label{eq:lame}
\left\{\begin{array}{lll}\d_tu-\mu\Delta u-(\lambda+\mu)\nabla\div u=f,\\
u_{|t=0}=u_0,\end{array}
\right.\end{equation}
where $\lambda$ and $\mu$ are constant coefficients such that $\mu>0$ and $\lambda+2\mu>0,$
also fulfill  \eqref{R-E14} (up to the dependence w.r.t.  the viscosity).
Indeed: both $\cP u$ and $\cQ u$ satisfy  a heat equation,
as may be easily seen by applying $\cP$ and $\cQ$ to \eqref{eq:lame}.
\end{rem}

\end{appendix}

\end{document}